\documentclass[12pt]{amsart}

\newtheorem{theorem}{Theorem}[section]
\newtheorem{lemma}[theorem]{Lemma}

\newtheorem{proposition}[theorem]{Proposition}

\newtheorem{corollary}[theorem]{Corollary}

\usepackage[bookmarks]{hyperref}
\usepackage{fullpage}
\usepackage{cite}
\usepackage{babel,graphicx,amscd,amsthm,amsfonts,amsopn,amssymb,verbatim,enumerate,xcolor}

\usepackage[letterpaper,left=1in,top=1in,right=1in,bottom=1in]{geometry}

\def\fracsa{\frac{s}{\alpha}}
\def\clogd{c(\alpha)|D_x|^{\alpha-1}}
\def\logd{|D_x|^{\alpha - 1}}

\def\tv{\tilde v}
\def\tw{\tilde w}
\def\half{ \frac{1}{2}}
\def\D{\partial}

\def\R{{\mathbb R}}

\def\nint{\mathop{\diagup\kern-13.0pt\int}}
\def\Z{{\mathbb Z}}

\def\bas{\begin{align*}}
\def\eas{\end{align*}}
\def\bi{\begin{itemize}}
\def\ei{\end{itemize}}

\def\emph#1{{\it #1}}

\def\eps{{\epsilon}}

\def\dq{{\delta}}
\def\sdq{{|\delta|}}

\def\pax{\mathbf{u}}
\def\rpax{\mathbf{r}}

\DeclareMathOperator{\sgn}{sgn}

\theoremstyle{definition}
\newtheorem{remark}[theorem]{Remark}
\numberwithin{equation}{section}

\title{Low regularity well-posedness for the generalized surface quasi-geostrophic front equation}

\author{Albert Ai}
\address{Department of Mathematics, University of Wisconsin, Madison}
\email{aai@math.wisc.edu}

\author{Ovidiu-Neculai Avadanei}
\address{Department of Mathematics, University of California at Berkeley}
\email{ovidiu\_avadanei@berkeley.edu}

\begin{document}

\begin{abstract}
We consider the well-posedness of the generalized surface quasi-geostrophic (gSQG) front equation. By using the null structure of the equation via a paradifferential normal form analysis, we obtain balanced energy estimates, which allow us to prove the local well-posedness of the non-periodic  gSQG front equation at a low level of regularity (in the SQG case, at only one-half derivatives above scaling). In addition, we establish global well-posedness for small and localized rough initial data, as well as modified scattering, by using the testing by wave packet approach of Ifrim-Tataru.
\end{abstract}


\maketitle
\addtocontents{toc}{\protect\setcounter{tocdepth}{1}}
\tableofcontents

\maketitle

\section{Introduction}
The generalized surface quasi-geostrophic (gSQG) equations are a one parameter family of active scalar equations parameterized by a transport term, given by
\begin{equation}\label{rawgSQG}
   \theta_t+u\cdot\nabla\theta=0,\qquad u=(-\Delta)^{-1 + \frac{\alpha}{2}} \nabla^\perp\theta, \qquad \alpha \in [0,2).
\end{equation}
Here, $\theta$ represents a scalar evolution on $\mathbb{R}^2$, $(-\Delta)^{-1 + \frac{\alpha}{2}}$ is a fractional Laplacian, and $\nabla^\perp$ is given by $\nabla^\perp=(-\partial_y,\partial_x)$.

The case $\alpha = 0$ corresponds to the two-dimensional incompressible Euler equation, while the case $\alpha = 1$ gives the surface quasi-geostrophic equation (SQG) equation. The latter arises as a model for quasi-geostrophic flows confined to a surface in atmospheric and oceanic science. It also shares some similarities with the three dimensional incompressible Euler equation, and thus is often used as a simplified model problem. In particular, the question of finite time singularity formation remains open for both equations.

Front solutions to \eqref{rawgSQG} are solutions of the form
\[
\theta(t, x, y) = \begin{cases}
  \theta_+  \qquad \text{if } y > \varphi(t, x), \\
  \theta_-  \qquad \text{if } y < \varphi(t, x),
\end{cases}
\]
where the front refers to the graph $y = \varphi(t, x)$ with $x \in \R$. Front solutions are closely related to patch solutions, which have the form
\[
\theta(t, x, y) = \begin{cases}
  \theta_+  \qquad \text{if } (x, y) \in \Omega(t), \\
  \theta_-  \qquad \text{if } (x, y) \notin \Omega(t),
\end{cases}
\]
where $\Omega(t)$ is a bounded, simply connected domain. 

When $\alpha \in (1, 2)$, the derivation and analysis of contour dynamics equations governing fronts and patches do not differ substantially. However, when $\alpha \in [0, 1]$, the derivation of the equations for fronts has additional complexities relative to the case of patches, arising from the slow decay of Green's functions. The derivation in this range was provided by Hunter-Shu \cite{HSderivation} via a regularization procedure, and again by Hunter-Shu-Zhang in \cite{HSZderivation}. In the generalized $\alpha \neq 1$ case, the equation takes the form
\begin{equation}\label{gSQG}
\begin{aligned}
(\partial_t - c(\alpha)|D_x|^{\alpha-1}\D_x)\varphi &= Q(\varphi , \D_x\varphi), \\
\varphi(0,x)&=\varphi_0(x),
\end{aligned}
\end{equation}
while in the SQG case $\alpha=1$, the equation
takes the form
\begin{equation}\label{SQG}
\begin{aligned}
(\partial_t - 2\log|D_x|\D_x)\varphi &= Q(\varphi , \D_x\varphi), \\
\varphi(0,x)&=\varphi_0(x).
\end{aligned}
\end{equation}
Here, $\varphi$ is a real-valued function $\varphi : [0, \infty) \times \R \rightarrow \R$, $c(\alpha)$ denotes the constant
\begin{equation}
\begin{aligned}
c(\alpha) &=-2\sin\left(\frac{\pi(2-\alpha)}{2}\right)\Gamma(1-\alpha),\qquad &&\alpha\in(0,2), \\
c(\alpha) &= -\half, \qquad &&\alpha = 0, \\
\end{aligned}
\end{equation}
and the nonlinearity $Q$ is given in the two cases $\alpha \in (0, 2)$ and $\alpha = 0$ respectively by
\begin{equation}\label{A-op}
Q(f,g)(x) = \int \left(\frac{1}{|y|^{\alpha}} - \frac{1}{(y^2+(f(x+y)-f(x))^2)^{\frac{\alpha}{2}}}\right)\cdot (g(x + y) - g(x)) \,dy
\end{equation}
and
\begin{equation}\label{A-op-zero}
Q(f,g)(x) = \int \frac{1}{2\pi}\log\left(1+\left(\frac{f(x+y)-f(x)}{y}\right)^{2}\right) \cdot (g(x+y)-g(x))\,dy.
\end{equation}

The nonlinearity $Q$ can be written more succinctly using difference quotients,
\begin{equation}\label{A-op2}
Q(f,g)(x) = \int \frac{1}{|y|^{\alpha-1}}F(\dq^yf) \cdot \sdq^yg\,dy,
\end{equation}
where
\[
\displaystyle \dq^yf(x)=\frac{f(x+y)-f(x)}{y}, \qquad \displaystyle \sdq^yg(x)=\frac{g(x+y)-g(x)}{|y|},
\]
and
\[
\displaystyle F(s)=1-\frac{1}{(1+s^2)^{\frac{\alpha}{2}}} \quad \text{when }\alpha \in (0, 2),\qquad \displaystyle F(s)=\frac{1}{2\pi}\log(1+s^2) \quad \text{when }\alpha=0.
\]

The equations \eqref{gSQG} and \eqref{SQG} are invariant under the scaling
\begin{align*}
    t\rightarrow \kappa^\alpha t,\qquad x\rightarrow\kappa x,\qquad\varphi\rightarrow\kappa\varphi,
\end{align*}
which implies that $\dot{H}^{\frac{3}{2}}(\mathbb{R})$ is the corresponding critical Sobolev space.

\

In the case of SQG patches, Gancedo-Nguyen-Patel proved in \cite{GNP} that under a suitable parametrization, the contour dynamics evolution is locally well-posed in $H^s(\mathbb{T})$ when $s>2$. The gSQG case with $\alpha\in(0,2)$ and $\alpha \neq 1$, was also considered by Gancedo-Patel \cite{GP}, where they in particular showed local well-posedness in $H^2$ for $\alpha \in (0, 1)$ and $H^3$ for $\alpha \in (1, 2)$. For a more recent result on enhanced lifespan for $\alpha$-patches, see Berti-Cuccagna-Gancedo-Scrobogna \cite{BCGS}. 

Patches in the 2D Euler case $\alpha=0$ correspond to a special kind of Yudovich \cite{Yudovich} solution, also referred to as Euler vortex patches. Bertozzi \cite{Bertozzi} showed that they are locally well-posed in the space $C^{1,\delta}$. By using paradifferential calculus, Chemin \cite{Chemin}  proved that given a patch solution whose boundary is $C^{k,\mu}$ at the initial time, its regularity persists for all times. Bertozzi-Constantin \cite{Bertozzi-Constantin} soon obtained another proof based on a level set approach, and yet another proof was subsequently obtained by Serfati \cite{Serfati}. Recently, Radu \cite{Radu} proved global existence for Euler patch solutions. For results in Sobolev spaces, see Coutand-Shkoller \cite{Coutand-Shkoller}. 

For the question of ill-posedness in the context of patches, Kiselev-Luo \cite{KL} obtained some results in Sobolev spaces with exponents $p \neq 2$, as well as in H\"older spaces. In addition, Zlato\v s \cite{Zlatos} showed that, provided local well-posedness is known for some $\alpha\in(0,\frac12]$ , suitable initial data give rise to blow up for both bounded and unbounded patch solutions.

\

In the current paper, we are interested in the well-posedness of gSQG fronts, following our previous work on the SQG case \cite{SQGzero2}. The first results in the gSQG setting were obtained by Córdoba-Gómez-Serrano-Ionescu \cite{GlobalPatch} for $\alpha \in (1,2)$, showing global well-posed for small and localized initial data in $H^s$, where $s > 20\alpha$. 

For the cases $\alpha \in (0, 1]$ spanning 2D Euler to the classical SQG, Hunter-Shu-Zhang first studied the local well-posedness for a cubic approximation of the SQG equation in \cite{HSZapprox}, before establishing in \cite{HSZglobal} local well-posedness for the full SQG equation \eqref{SQG} with initial data in $H^s$ with $s \geq 5$, along with global well-posedness for small, localized, and essentially smooth ($s \geq 1200$) initial data. These results were later extended to the gSQG cases $\alpha \in (0, 1)$, while also lowering the regularity threshold to $s > \frac72 + \frac{3\alpha}{2}$ \cite{HSZfamily}.

Although it is typical to study global well-posedness in the context of small and localized data, we remark that the local well-posedness results of \cite{HSZapprox, HSZfamily} were also established in the context of a small data assumption, as well as a convergence condition on an expansion of the nonlinearity $Q(\varphi, \D_x\varphi)$ from \eqref{SQG}. There, the purpose of the small data was to ensure that the modified energies used in the proof were coercive.

In the SQG case $\alpha = 1$, the authors lowered both the local and global well-posedness regularity thresholds, to $s>\frac{5}{2}$ and $s>4$ respectively \cite{SQGzero}, while removing the small data and convergence conditions from the local well-posedness result. Subsequently, by observing a null structure satisfied by the SQG equation \eqref{SQG}, the authors further improved the low regularity thresholds to $s>2$ and $s>3$ respectively \cite{SQGzero2}, while also improving the low frequency threshold to $\dot{H}^{s_0}$ for any $s_0<\frac{3}{2}$. In particular, this result establishes local well-posedness without requiring control at the level of $L^2$.

Our current objective is to consider the generalized SQG family $\alpha \in [0,2)$, to prove lower regularity local and global well-posedness results which parallel those of the SQG case $\alpha = 1$. Our contributions include
\begin{itemize}
\item obtaining a local well-posedness result in a significantly lower regularity setting, at $\frac{\alpha}{2}$ derivatives above scaling, by making use of a null structure exhibited by the generalized family of equations, and
\item proving low regularity global well-posedness by using the wave packet testing method of Ifrim-Tataru (see for instance \cite{ITschrodinger, ITpax}).
\end{itemize}

\subsection{Main results}
Similar to the SQG setting considered in \cite{SQGzero2}, a key observation of the current article is that the gSQG equation \eqref{gSQG} exhibits a resonance structure when $\alpha\in(0,1)\cup(1,2)$. More precisely, we can approximate the nonlinearity $Q$ by
\begin{equation}
    Q(\varphi,v)\approx\Omega(\psi, v), \qquad \psi := \partial_x^{-1}F(\varphi_x)
\end{equation}
where $\Omega$ is a bilinear form whose symbol is given by the resonance function
\begin{equation*}
\Omega(\xi_1, \xi_2)=\frac{1}{\alpha}(\omega(\xi_1) + \omega(\xi_2) - \omega(\xi_1 + \xi_2)), \qquad \omega(\xi) = c(\alpha)i \xi|\xi|^{\alpha-1}.
\end{equation*}
This is meaningful because $\psi$ solves an equation similar to $\varphi$ (see Proposition ~\ref{p:psi-eqn}, part $b)$).

This structure enables us to use normal form methods to obtain \emph{balanced energy estimates}, which use control norms with an even balance of derivatives,
\[
\frac{d}{dt}E^{(s)}(\varphi)\lesssim_A B^2 \cdot E^{(s)}(\varphi),
\]
where
\begin{equation}\label{control-norms}
A = \|\D_x \varphi\|_{L^\infty}, \qquad B=\|\D_x \varphi\|_{B_{\infty,2}^{\frac{\alpha}{2}} \cap BMO^{\frac{\alpha}{2}}}.
\end{equation}

We can now state our main local well-posedness result:

\begin{theorem}\label{t:lwp}
Let $\alpha\in[0,2)$. Equation \eqref{gSQG} is locally well-posed for initial data in $\dot H^{s_0} \cap \dot H^s$ with $s_0<\frac32$ and
\begin{equation*}
\begin{aligned}
&s > \frac{\alpha + 3}{2} \qquad &&\text{if} \quad \alpha = 0, 1, \\
&s\geq\frac{\alpha+3}{2} \qquad &&\text{if} \quad \alpha\in(0,1)\cup(1,2).
\end{aligned}
\end{equation*}
Precisely, for every $R > 0$, there exists $T=T(R)>0$ such that for any $\varphi_0\in \dot H^{s_0} \cap \dot H^s$ with $\|\varphi_0\|_{\dot H^{s_0} \cap \dot H^s} < R$, the Cauchy problem \eqref{SQG} has a unique solution $\varphi \in C([0, T], \dot H^{s_0} \cap \dot H^s)$. Moreover,  the solution map $\varphi_0 \mapsto \varphi$ from $\dot H^{s_0} \cap \dot H^s$ to $C([0, T], \dot H^{s_0} \cap \dot H^s)$ is continuous.
\end{theorem}
\begin{remark}
The SQG case $\alpha=1$ is the subject of our previous paper \cite{SQGzero2}, while the 2D Euler case $\alpha = 0$ is addressed in Appendix~\ref{s:appendix}.
    
For the main subject of the current paper, the cases $\alpha \in (0, 1) \cup (1, 2)$, the control norms \eqref{control-norms} allow us to obtain the local well-posedness in the endpoint case $s = \frac{\alpha+3}{2}$. In contrast, in the SQG and 2D Euler cases, we only prove the result in the case $s>2$ and $s > 3/2$ respectively, due to the logarithmic loss generated by the dispersion relation and nonlinearities. Precisely, for comparison, we recall that in the analysis of the SQG case in \cite{SQGzero2}, we used the control parameter
\[
B = \|\D_x \varphi\|_{C^{\frac{1}{2}+}}.
\]

\end{remark}
\begin{remark}
    For the sake of simplicity, we assume that the parameter $A$ is small. This assumption can be removed with a more careful definition of the paraproduct, at the expense of having to deal with more technical details; see for instance \cite{SQGzero}.
\end{remark}

Normal forms were introduced by Shatah \cite{Shatah}, who used them to prove results on the long-time dynamics of solutions to dispersive equations. Unfortunately, this method cannot be applied directly to quasilinear problems, because the resulting transformations would be unbounded. To address this issue, several approaches have been introduced; we rely on two in the current paper. The first consists of carrying out the analysis in a paradifferential manner, and was introduced by Alazard-Delort \cite{Alazard-Delort} in a paradiagonalization argument to obtain Sobolev estimates for the solutions of the water waves equations in the Zakharov formulation. This approach was also subsequently employed by Ifrim-Tataru \cite{Benjamin} to obtain a new proof of $L^2$ global well-posedness for the Benjamin-Ono equation, a result first obtained in \cite{Ionescu-Kenig}.

The second method consists of employing modified energies instead of applying the direct normal form at the level of the equation. This procedure was introduced by Hunter-Ifrim-Tataru-Wong \cite{Hunter-Ifrim-Tataru-Wong} to study the long time behavior of solutions to the Burgers-Hilbert equation.

These approaches were first combined in order to study the low regularity well-posedness for quasilinear models by the first author together with Ifrim-Tataru \cite{Ai-Ifrim-Tataru} for the gravity water waves system, through the proof of \textit{balanced energy estimates}. This type of estimate was then later used together with Strichartz estimates to obtain low regularity well-posedness results for the time-like minimal surface problem in the Minkowski space \cite{Minimal-surface}.

\

We next consider global well-posedness for small and localized data. To describe localized solutions, we define the operator 
\[
L = x+t\alpha c(\alpha)|D_x|^{\alpha-1},
\]
which commutes with the linear flow $\partial_t-c(\alpha)|D_x|^{\alpha-1}\D_x$, and at time $t = 0$ is simply multiplication by $x$. Then we define the time-dependent weighted energy space
\[
\|\varphi\|_X := \|\varphi\|_{\dot{H}^{s_0}\cap\dot{H}^s} + \|L\partial_x \varphi\|_{L^2},
\]
where $s>\alpha+2$ and $s_0<1$. To track the dispersive decay of solutions, we define the pointwise control norm 
\[
\|\varphi\|_Y := \||D_x|^{1-\delta}\langle D_x\rangle^{\frac{\alpha}{2}+2\delta}\varphi\|_{L_x^{\infty}}.
\]

\begin{theorem}\label{t:gwp}
Consider data $\varphi_0$ with
\[
\|\varphi_0\|_X \lesssim \eps \ll 1.
\]
Then the solution $\varphi$ to \eqref{gSQG} for $\alpha > 0$ with initial data $\varphi_0$ exists globally in time, with energy bounds 
\[
\|\varphi(t)\|_{X} \lesssim \eps t^{C\eps^2}
\]
and pointwise bounds 
\[
\|\varphi(t)\|_Y \lesssim \eps \langle t\rangle^{-\half}.
\]
\end{theorem}
 Further, the solution  $\varphi$ exhibits a modified scattering behavior, with an asymptotic profile $W\in H^{1-C_1\epsilon^2}(\R)$, in a sense that will be made precise in Section~\ref{s:scattering}.
 
 \subsection{Outline of the paper.} The organization of the paper follows closely that of the SQG case \cite{SQGzero2}. Here, the main novelties include improvements which allow us to obtain the endpoint regularity $s = \frac{\alpha + 3}{2}$, as well as a generalized modified energy and normal form transformation. 
 
 In Section~\ref{s:notation}, we introduce notations and establish preliminary estimates for paradifferential calculus and difference quotients.
 
In Section~\ref{s:equations}, we describe the null structure of equation \eqref{gSQG} and of its linearization,
\begin{equation}\label{linearized-eqn}
\D_tv - c(\alpha)|D_x|^{\alpha-1} \D_xv = \D_x Q(\varphi, v).
\end{equation}
We also define the paradifferential flow associated to \eqref{linearized-eqn}, which will play a key role in the analysis.

In Section ~\ref{s:reduction} we reduce the energy estimates and well-posedness for the linearized equation \eqref{linearized-eqn} to the case of the inhomogeneous paradifferential flow. The main challenge is ensuring that the paradifferential error terms satisfy cubic balanced energy estimates, which we achieve by performing a normal form analysis. 

In Section~\ref{s:energy estimates} we prove energy estimates for the paradifferential flow. To achieve this, we construct a paradifferential modified energy functional. Here, we note that unlike in the SQG case $\alpha=1$ \cite{SQGzero2}, the normal form underlying the energy functional coincides with the normal form correction of the previous section only to first order, the latter being just a linearization of the former.

In Section~\ref{s:higher order energy estimates}, we obtain higher order energy estimates. As the resulting commutators are quadratic and thus not perturbative, they pose an additional obstacle. To address this, we perform a two-step change of variables. First, we eliminate the highest order terms by using a Jacobian exponential conjugation, at the cost of creating a number of lower order factors. Then, we carry out a normal form analysis to eliminate these remaining non-perturbative terms.

In Section~\ref{s:lwp}, we provide the proof of our local well-posedness result. Here we construct rough solutions as the unique limits of smooth solutions, controlled by frequency envelopes. This technique was introduced by Tao in \cite{25}, who used it to obtain more accurate information on the evolution of the energy distribution between dyadic frequencies under nonlinear flows. This method is systematically presented in the context of local well-posedness theory for quasilinear problems by Ifrim-Tataru in their expository paper \cite{ITprimer}. 

In Section~\ref{s:gwp} we prove the global well-posedness part of Theorem~\ref{t:gwp} and the dispersive bounds on the resulting solutions, by using the wave packet testing method of Ifrim-Tataru \cite{ITpax}. In Section~\ref{s:scattering}, we prove the modified scattering behavior, completing the proof of Theorem \ref{t:gwp}.

Finally, in the appendix we prove the local well-posedness result in the case $\alpha=0$, which corresponds to Euler fronts. Here, we note that due to the degenerate character of the dispersion relation, the methods that we used cannot be applied to obtain an analogue of the global well-posedness result.

\subsection{Acknowledgements}

The first author was supported by the NSF grant DMS-2220519 and the RTG in Analysis and Partial Differential equations grant DMS-2037851. The second author was supported by the NSF
grant DMS-2054975, as well as by the Simons Foundation. 

The authors were also supported by the NSF under Grant No. DMS-1928930 while in residence at the Simons Laufer Mathematical Sciences Institute (formerly MSRI) in Berkeley, California, during the summer of 2023.

The authors would like to thank Mihaela Ifrim and Daniel Tataru for many helpful discussions.

\section{Notations and classical estimates}\label{s:notation}

In this section we introduce some notations and classical estimates that we use throughout the article. These include paradifferential calculus and difference quotient estimates.

\subsection{Paradifferential operators and paraproducts}

Let $\chi$ be an even smooth function such that $\chi=1$ on $[-\frac{1}{20}, \frac{1}{20}]$ and $\chi = 0$ outside $[-\frac{1}{10}, \frac{1}{10}]$, and define
\[
\tilde{\chi}(\theta_1,\theta_2) = \chi\left(\frac{|\theta_1|^2}{M^2+|\theta_2|^2}\right).
\]
Given a symbol $a(x,\eta)$,  we use the above cutoff $\tilde \chi$ to define an $M$-dependent paradifferential quantization of $a$ by (see also \cite{ABZgravity}) 
\begin{align*}
    \widehat{T_au}(\xi)=(2\pi)^{-1}\int \hat{P}_{>M}(\xi)\tilde \chi\left(\xi - \eta, \xi + \eta \right) \hat{a}(\xi-\eta,\eta)\hat{P}_{>M}(\eta)\hat{u}(\eta)\,d\eta,
\end{align*}
where the Fourier transform of the symbol $a(x,\eta)$ is taken with respect to the first argument.

This quantization was employed in \cite{SQGzero}, where the parameter $M$ was introduced to ensure the coercivity of the modified quasilinear energy without relying on a small data assumption. We recall in particular that in the case of a paraproduct, where $a = a(x)$, $T_a$ is self-adjoint.

\

The following two commutator-type estimates are exact reproductions of statements from Lemmas 2.4 and 2.6 in Section 2 of \cite{Ai-Ifrim-Tataru}, respectively:

\begin{lemma}[Para-commutators]\label{l:para-com}
 Assume that $\gamma_1, \gamma_2 < 1$. Then we have
\begin{equation}\label{para-com}
\| T_f T_g - T_g T_f \|_{\dot H^{s} \to \dot H^{s+\gamma_1+\gamma_2}} \lesssim 
\||D|^{\gamma_1}f \|_{BMO}\||D|^{\gamma_2}g\|_{BMO},
\end{equation}
\begin{equation}
\| T_f T_g - T_g T_f \|_{\dot B^{s}_{\infty,\infty} \to \dot H^{s+\gamma_1+\gamma_2}} \lesssim 
\||D|^{\gamma_1}f \|_{L^2}\||D|^{\gamma_2}g\|_{BMO}.
\end{equation}
A bound similar to \eqref{para-com} holds in the Besov scale of spaces, namely 
from $\dot B^{s}_{p, q}$ to $\dot B^{s+\gamma_1+\gamma_2}_{p, q}$
for real $s$ and $1\leq p,q \leq \infty$.
\end{lemma}

The next paraproduct estimate, see \cite[Lemma 2.5]{Ai-Ifrim-Tataru}, directly relates multiplication and paramultiplication:

\begin{lemma}[Para-products]\label{l:para-prod}
Assume that $\gamma_1, \gamma_2 < 1$. Then
\begin{equation}
\| T_f T_g - T_{T_fg} \|_{\dot H^{s} \to \dot H^{s+\gamma_1+\gamma_2}} \lesssim 
\||D|^{\gamma_1}f \|_{BMO}\||D|^{\gamma_2}g\|_{BMO}.
\end{equation}
If in addition $\gamma_1+\gamma_2 \geq 0$,
\begin{equation}
\| T_f T_g - T_{fg} \|_{\dot H^{s} \to \dot H^{s+\gamma_1+\gamma_2}} \lesssim 
\||D|^{\gamma_1}f \|_{BMO}\||D|^{\gamma_2}g\|_{BMO}.
\end{equation}

Similar bounds hold in the Besov scale of spaces, namely 
from $\dot B^{s}_{p, q}$ to $\dot B^{s+\gamma_1+\gamma_2}_{p, q}$
for real $s$ and $1\leq p,q \leq \infty$.
\end{lemma}

Next, we recall the following Moser-type estimate; see for instance \cite{SQGzero}.
\begin{theorem}[Moser]\label{t:moser}
Let $F:\R \rightarrow \R$ be a smooth function with $F(0) = 0$, $\gamma_1+\gamma_2 =s>0$, and 
\[
R(v) = F(v) - T_{F'(v)} v.
\]
Then 
\begin{equation}\label{moser}
\|R(v)\|_{\dot{W}^{s, \infty}}\lesssim_{\|v\|_{L^\infty}} \||D|^{\gamma_1} v\|_{BMO\cap \dot{B}^{0}_{\infty,2}}\||D|^{\gamma_2} v\|_{BMO\cap \dot{B}^{0}_{\infty,2}}.
\end{equation}
\end{theorem}

\

From this Moser estimate, we obtain the following Moser-type paraproduct estimate:
\begin{lemma}\label{l:para-prod3}
Assume that $0 \leq \gamma_1+\gamma_2 < 1$. Let $F$ be a smooth function. Then we have
\begin{equation}
\| T_{\D_xF(f)} - T_{F'(f)}T_{f_x}\|_{\dot H^{s} \to \dot H^{s+\gamma_1+\gamma_2 - 1}} \lesssim
\||D|^{\gamma_1}f \|_{BMO}\||D|^{\gamma_2}f\|_{BMO}.
\end{equation}
\end{lemma}

\begin{remark}
   We need this estimate only in the $\alpha < 1$ case, to shift derivatives between both the low frequency paraproducts as well as the high frequency argument. In the $\alpha > 1$ case, it always suffices to balance derivatives between just the low frequency paraproducts using Lemma~\ref{l:para-prod} directly.
\end{remark}

\begin{proof}
We apply the Moser estimate of Theorem~\ref{t:moser} to replace $\D_xF(f)$ with $\D_x T_{F'(f)} f$. A second application of a Moser estimate then estimates the contribution from $T_{\D_x F'(f)} f$, which leaves us with
\begin{align*}
    T_{T_{F'(f)}f_x} - T_{F'(f)}T_{f_x}.
\end{align*}
By the first estimate of Lemma~\ref{l:para-prod}, it follows that this error is acceptable. 
\end{proof}

\subsection{Difference quotients}

Recall that we denote difference quotients by 
\[
\dq^yh(x)=\frac{h(x+y)-h(x)}{y}, \qquad \sdq^yh(x)=\frac{h(x+y)-h(x)}{| y|},
\]
and define the smooth function 
\[
F(s) = 1 - \frac{1}{(1+s^2)^{\frac{\alpha}{2}}},
\]
which in particular vanishes to second order at $s = 0$, satisfying $F(0)=F'(0)=0$. Using this notation, we may express
\[
    Q(\varphi, v)(t,x) = \int |y|^{1 - \alpha} F(\dq^y\varphi(t,x)) \cdot \sdq^y v(t,x)\,dy.
\]
In addition, to facilitate the normal form analysis in later sections, we denote
\[
\psi := \D_x^{-1} F(\D_x \varphi), \qquad J = (1 - F(\D_x \varphi))^{-1} = (1 - \D_x \psi)^{-1},
\]
where we fix the antiderivative such that $\psi(-\infty) = 0$.

We have the following estimate which allows the balancing of up to $2 - \alpha$ derivatives over multilinear averages of difference quotients:

\begin{lemma}\label{Trilinear integral estimate-v0}
Let $i = \overline{1,n}$ and $p_i, r \in [1, \infty]$ and $\alpha_i, \beta_i \in [0, 1]$ satisfying
\[
\sum_{i}\frac{1}{p_i} = \frac{1}{r}, \qquad n+\alpha-2 < \sum_{i}\alpha_i\leq n, \qquad 0 \leq \sum_{i}\beta_i < n+\alpha-2.
\]
Then
\[
 \left\|\int\frac{1}{|y|^{\alpha-1}} \prod \dq^yf_i \,dy\right\|_{L_x^r} \lesssim \prod \||D|^{\alpha_i} f_i\|_{L^{p_i}} + \prod \||D|^{\beta_i} f_i\|_{L^{p_i}}.
 \]

\end{lemma}

\begin{proof}
We write 
\[
\int \frac{1}{|y|^{\alpha-1}}\prod \dq^yf_i \,dy = \int_{|y| \leq 1} + \int_{|y| > 1}.
\]
For the former integral, we have by H\"older
\[
 \left\| \int_{|y| \leq 1}\frac{1}{|y|^{\alpha-1}} \prod \dq^yf_i \,dy \right\|_{L_x^r} \lesssim \int_{|y| \leq 1} \frac{1}{|y|^{n+\alpha-1 - \sum \alpha_i}} \prod \||D|^{\alpha_i} f_i\|_{L^{p_i}} \, dy \lesssim \prod \||D|^{\alpha_i} f_i\|_{L^{p_i}}.
\]
The latter integral is treated similarly.
\end{proof}

The above lemma may be sharpened using Besov spaces. First, we recall the following difference quotient representation of the Besov space $\dot{B}^s_{p,r}$ \cite[Theorem 2.36]{Bahouri-Chemin}:

\begin{lemma}\label{Endpoint integral estimate}
Let $s\in(0,1)$ and $(p,r) \in [1, \infty]^2$. Then
\[
\|u\|_{\dot{B}^s_{p,r}} \approx \left\|\frac{\|u(x+y)-u(x)\|_{L_x^p}}{|y|^s}\right\|_{L_y^r(\mathbb{R},\frac{1}{|y|})}.
 \]
\end{lemma}
We apply this to establish a Besov version of the estimate on multilinear averages:
\begin{lemma}\label{Besov trilinear integral estimate}
Let $i = \overline{1,n}$ and $p_i, r \in [1, \infty]$ and $\alpha_i\in [0, 1)$ satisfying
\[
\sum_{i}\frac{1}{p_i} = \frac{1}{r},\qquad \sum_{i}\frac{1}{q_i}=1, \qquad \sum_{i}\alpha_i=2-\alpha.
\]
Then
\[
 \left\|\int\frac{1}{|y|^{\alpha-1}} \prod \dq^yf_i \,dy\right\|_{L_x^r} \lesssim \prod_{i=1}^n \|f_i\|_{\dot{B}^{1-\alpha_i}_{p_i,q_i}}.
 \]

\end{lemma}

\begin{proof}
By the triangle and H\"older's inequalities, we have
\begin{align*}
 \left\|\int\frac{1}{|y|^{\alpha-1}} \prod \dq^yf_i \,dy\right\|_{L_x^r} &\lesssim  \int\frac{1}{|y|} \left\|\prod |y|^{\alpha_i}\dq^yf_i\right\|_{L_x^r} \,dy\lesssim \prod_{i=1}^n\int\frac{1}{|y|}\left\| |y|^{\alpha_i}\dq^yf_i\right\|_{L_x^{p_i}} \,dy\\
 &\lesssim \prod_{i=1}^n \left\|\left\| |y|^{\alpha_i}\dq^yf_i\right\|_{L_x^{p_i}}\right\|_{L_y^{q_i}\left(\mathbb{R},\frac{1}{|y|}\right)}\lesssim \prod_{i=1}^n \|f_i\|_{\dot{B}^{1-\alpha_i}_{p_i,q_i}}.
\end{align*}

\end{proof}

\section{The null structure and paradifferential equation}\label{s:equations}

In this section and the next, we reduce energy estimates for the linearized equation \eqref{linearized-eqn},
\[
\D_tv - c(\alpha)|D_x|^{\alpha-1} \D_xv = \D_x Q(\varphi, v),
\]
to energy estimates for a corresponding paradifferential equation. 

This can be achieved by treating \eqref{linearized-eqn} as a paradifferential equation with a perturbative source, where the main task is to paralinearize the cubic term $\D_x Q(\varphi, v)$. Moreover, we are interested in carrying out this process in a manner such that the perturbative errors satisfy \emph{balanced estimates}. Precisely, we obtain estimates that only involve the control norms
\[
A := \|\D_x \varphi\|_{L^\infty}, \qquad B := \|\D_x \varphi\|_{B_{\infty,2}^{\frac{\alpha}{2}} \cap BMO^{\frac{\alpha}{2}}},
\]
where $A$ corresponds to the scaling-critical threshold, while $B$ lies $\alpha/2$ derivatives above scaling.

Unfortunately, directly estimating the paralinearization errors does not allow us to prove estimates that are controlled only by $A$ and $B$. Instead, we will use a paradifferential normal form transformation to eliminate the source terms that do not directly satisfy the desired balanced cubic estimates. In this section, we first consider various formulations of the paradifferential equation which will be useful in the following sections.

\subsection{Null structure}

Even though the principal term in the expansion of $F(\dq^y \varphi)$ is quadratic in $\varphi$ (and thus $Q(\varphi, v)$ is principally cubic), estimates on derivatives of $F(\dq^y \varphi)$ do not fully capture its quadratic structure. This happens because they are limited by the cases of low-high interaction where derivatives fall solely on the high frequency variable. Consequently, in the context of proving balanced estimates, $F(\dq^y \varphi)$ behaves essentially like a linear coefficient. 

On the other hand, we note that $Q$ exhibits a null structure in the following sense. By writing
\[
\Omega(\varphi, v) = \int \frac{1}{|y|^{\alpha-1}}\dq^y\varphi \cdot \sdq^yv\,dy
\]
and using the heuristic approximation
\[
F(\dq^y \varphi) \approx T_{F'(\varphi_x)} \dq^y \varphi,
\]
we may express $Q$ as a quadratic form with a low frequency coefficient,
\begin{equation}\label{atoq}
Q(\varphi, v) \approx T_{F'(\varphi_x)}\Omega(\varphi, v).
\end{equation}
We then observe that the bilinear form $\Omega(\varphi, v)$ exhibits a null structure, since its symbol (abusing notation)
\[
\Omega(\xi_1, \xi_2) = \int \frac{\sgn{y}}{|y|^{\alpha-1}} \cdot \frac{(e^{i\xi_1 y}-1)(e^{i\xi_2 y}-1)}{y^2} \, dy
\]
satisfies the resonance identity
\begin{equation}\label{resonance}
\alpha\Omega(\xi_1, \xi_2) = \omega(\xi_1) + \omega(\xi_2) - \omega(\xi_1 + \xi_2), \qquad \omega(\xi) = c(\alpha)i \xi |\xi|^{\alpha-1}.
\end{equation}
This null structure is crucial for the normal form analysis, which we carry out in the next section.

\

We make the above discussion precise in the following lemma. Recall that we denote
\[
\psi := \D_x^{-1} F(\varphi_x).
\]

\begin{lemma}\label{l:toy-red}
We have
\[
Q(\varphi, v) = \Omega(\psi, v) + R(x, D) v
\]
where 
\[
\|(\D_x R)(x, D) v\|_{L^2} \lesssim_A B^2 \|v\|_{L^2}.
\]
\end{lemma}

\begin{proof}
We write
\begin{equation}
\begin{aligned}
Q(\varphi, v) - \Omega(\psi, v) &= \int \frac{1}{|y|^{\alpha-1}}\frac{F(\dq^y \varphi) - \dq^y \D_x^{-1} F(\varphi_x)}{|y|} \cdot (v(x + y) - v(x)) \, dy
\end{aligned}
\end{equation}
and set
\[
r(x, \xi) = -\int \frac{1}{|y|^{\alpha-1}}\frac{F(\dq^y \varphi) - \dq^y \D_x^{-1} F(\varphi_x)}{|y|} (e^{i\xi y} - 1) \, dy.
\]

Then we have
\begin{equation*}
\begin{aligned}
(\D_x R)(x, D)v &= \int \frac{1}{|y|^{\alpha-1}}\frac{F'(\dq^y \varphi)\dq^y \varphi_x - \dq^y F(\varphi_x)}{|y|} \cdot (v(x + y) - v(x)) \, dy \\
& =: \int K(x, y) \cdot (v(x + y) - v(x)) \, dy.
\end{aligned}
\end{equation*}

We first estimate $K$, which we may write as 
\begin{equation*}
\begin{aligned}
|y|^{\alpha} K(x, y) &= \frac{1}{y} (F'(b)(a - b) - (F(a) - F(b))+\frac{1}{y} (F'(\dq^y \varphi)-F'(\varphi_x))(a - b) \\
&=: |y|^{\alpha} K_1(x, y) + |y|^{\alpha} K_2(x, y),
\end{aligned}
\end{equation*}
where $a = \varphi_x(x + y)$, $b = \varphi_x(x)$. From $K_1$ we obtain a Taylor expansion,
\[
\|K_1(\cdot, y)\|_{L^\infty_x} \lesssim_A  \frac{1}{|y|^{\alpha-1}}\left\|\frac{a - b}{y} \right\|_{L^\infty_x}^2 = \frac{1}{|y|^{\alpha-1}}\|\dq^y \varphi_x\|_{L^\infty_x}^2.
\]
For $K_2$, we have 
\begin{align*}
\|K_2(\cdot, y) \|_{L^\infty_x} &\lesssim_A \frac{1}{|y|^{\alpha-1}}\left\|\frac{\varphi(x+y)-\varphi(x)-y\varphi_x(x)}{y^2} \right\|_{L^\infty_x} \left\|\frac{a - b}{y} \right\|_{L^\infty_x} \\
&=  \frac{1}{|y|^{\alpha-1}}\|\dq^{y, (2)} \varphi\|_{L^\infty_x} \|\dq^y \varphi_x\|_{L^\infty_x},
\end{align*}
where $\dq^{y, (2)}$ denotes the second-order difference quotient.

By Minkowski's inequality,
\begin{equation*}
\begin{aligned}
\|(\D_x R)(x, D)v \|_{L^2} &\lesssim \int \|K(\cdot, y)\|_{L^\infty_x} \|v(\cdot + y) - v(\cdot) \|_{L^2_x} \, dy \\
&\lesssim_A  \int \frac{1}{|y|^{\alpha-1}}\|\dq^y \varphi_x\|_{L^\infty_x} (\|\dq^y \varphi_x\|_{L^\infty_x} + \|\dq^{y, (2)} \varphi\|_{L^\infty_x} )\|v \|_{L^2_x} \, dy.
\end{aligned}
\end{equation*}
By Lemma \ref{Besov trilinear integral estimate}, we conclude
\[
\|(\D_x R)(x, D)v \|_{L^2} \lesssim_A B^2 \|v\|_{L^2}.
\]

\

\end{proof}

\subsection{The paradifferential flow}

We consider the linearized equation \eqref{linearized-eqn}, to which we associate its corresponding inhomogeneous paradifferential flow,
\begin{equation}\label{paradiff-eqn}
\D_tv - \clogd \D_xv - \D_x Q_{lh}(\varphi, v) = f,
\end{equation}
where the frequency decomposition of the (essentially) quadratic form has been expressed as
\begin{equation}\label{paralin}
\begin{aligned}
Q(\varphi, v) &= \int \frac{1}{|y|^{\alpha-1}}T_{F(\dq^y \varphi)} \sdq^yv\,dy + \int \frac{1}{|y|^{\alpha-1}}T_{\sdq^yv} F(\dq^y \varphi) \,dy + \int \frac{1}{|y|^{\alpha-1}}\Pi(\sdq^yv, F(\dq^y \varphi)) \,dy \\
&=: Q_{lh}(\varphi, v) + Q_{hl}(\varphi, v) + Q_{hh}(\varphi, v).
\end{aligned}
\end{equation}
We analogously decompose $\Omega = \Omega_{lh} + \Omega_{hl} + \Omega_{hh}$.

In Section~\ref{s:reduction}, we prove that the linearized equation \eqref{linearized-eqn} reduces to its paradifferential version \eqref{paradiff-eqn}, where $f$ is perturbative in the sense that it satisfies balanced, cubic estimates. On the other hand, as $Q$ and hence its paradifferential components $Q_{hl}(\varphi, v)$ and $Q_{hh}(\varphi, v)$ are essentially quadratic, we will first have to perform an appropriate paradifferential normal form change of variables for this to become apparent. 

Here, we prove a preliminary quadratic estimate for the reduction, which will be used when constructing and evaluating the contributions of the normal form transformation later in Section~\ref{s:reduction}.

\

We first extract the principal components of $\Omega_{lh}$, which include a transport term of order 0 and a dispersive term of  order $\alpha-1$:

\begin{lemma}\label{l:q-principal}
We can express
\begin{equation}\label{qlh-expression}
\alpha Q_{lh}(\varphi, v) = c(\alpha) (T_{|D_x|^{\alpha-1} \D_x \psi + R} v - T_{\D_x\psi} |D_x|^{\alpha-1} v + \D_x [T_\psi, |D_x|^{\alpha-1}] v).
\end{equation}
Further, we have
\begin{equation}\label{qlh-expression-full}
Q_{lh}(\varphi, v) = c(\alpha)(T_{\frac{1}{\alpha}(|D_x|^{\alpha-1}  \D_x \psi + R)} v - T_{\D_x \psi} |D_x|^{\alpha-1} v) + \Gamma(\D_x^2 \psi, |D_x|^{\alpha-2} v)
\end{equation}
where
\begin{equation}\label{l-est}
\||D_x|^{\frac{2-\alpha}{2}} \Gamma\|_{L^2} \lesssim_A B \|v\|_{L^2}, \qquad \|\D_x \Gamma\|_{L^2} \lesssim_A B \||D|^{\frac{\alpha}{2}} v\|_{L^2},
\end{equation}
as well the pointwise estimates
\begin{equation}\label{l-est-infty}
\begin{aligned}
&\||D_x|^{\frac{2-\alpha}{2}} \Gamma\|_{L^\infty} \lesssim_A B \|v\|_{L^\infty}, \qquad  &\|\D_x \Gamma\|_{L^\infty} \lesssim_A B \||D|^{\frac{\alpha}{2}} v\|_{L^\infty}.
\end{aligned}
\end{equation}
\end{lemma}

\begin{proof}
We use the resonance identity \eqref{resonance} to expand
\begin{equation}
\begin{aligned}
\alpha\Omega_{lh}(\psi, v) &= c(\alpha)(T_{|D_x|^{\alpha-1} \D_x \psi} v + [T_\psi, |D_x|^{\alpha-1} \D_x]v)\\
&= c(\alpha)(T_{|D_x|^{\alpha-1} \D_x \psi} v - T_{\D_x\psi} |D_x|^{\alpha-1} v + \D_x [T_\psi, |D_x|^{\alpha-1}] v).
\end{aligned}
\end{equation}
Combined with the low-high component of Lemma~\ref{l:toy-red}, we obtain \eqref{qlh-expression}. To then obtain \eqref{qlh-expression-full}, the remaining commutator may be expressed as
\[
-c(\alpha)(\alpha - 1) T_{\D_x \psi} |D_x|^{\alpha-1}v + \alpha \Gamma(\D_x^2 \psi, |D_x|^{\alpha-2} v)
\]
where $\Gamma$ denotes the subprincipal remainder, which has a favorable balance of derivatives on the low frequency and thus may be estimated as \eqref{l-est} and \eqref{l-est-infty}. 
\end{proof}

\begin{proposition}\label{p:quadratic-lin}
Consider a solution $v$ to \eqref{linearized-eqn}. Then $v$ satisfies
\begin{equation}\label{paradiff-eqn-2}
(\D_t - c(\alpha)(T_{J^{-1}} |D_x|^{\alpha-1} - T_{\frac{1}{\alpha}(|D_x|^{\alpha-1} \D_x\psi + R)})\D_x) v = f
\end{equation}
where
\begin{equation}\label{quadratic-est}
\||D|^{-\frac{2-\alpha}{2}}f\|_{L^2} \lesssim_A B \|v\|_{L^2}.
\end{equation}
\end{proposition}

\begin{proof}
We express \eqref{linearized-eqn} in terms of the paradifferential equation \eqref{paradiff-eqn} with source,
\[
\D_tv - \clogd \D_xv - \D_x Q_{lh}(\varphi, v) = \D_x Q_{hl}(\varphi, v) + \D_x Q_{hh}(\varphi, v).
\]

We estimate the source terms. We see directly from definition that $\D_x Q_{hl}(\varphi, v)$ has a favorable balance of derivatives which satisfies \eqref{quadratic-est} and may be absorbed into $f$. The balanced $Q_{hh}$ term similarly satisfies \eqref{quadratic-est}, so we have thus reduced \eqref{linearized-eqn} to \eqref{paradiff-eqn}.

It then suffices to apply \eqref{qlh-expression-full} of Lemma~\ref{l:q-principal} to the remaining paradifferential $Q_{lh}$ term on the left hand side of \eqref{paradiff-eqn} to obtain \eqref{paradiff-eqn-2}. Here, the $\Gamma$ contribution may be absorbed into $f$ directly. Further, we have commuted the $\D_x$ outside $Q_{lh}$ through the low frequency paracoefficients, since the cases where this derivative falls on the low frequency coefficients,
\[
c(\alpha)(T_{\frac{1}{\alpha}(|D_x|^{\alpha-1} \D_x^2 \psi + \D_x R)} v - T_{\D_x^2 \psi} |D_x|^{\alpha-1} v),
\]
have a favorable balance of derivatives, satisfying \eqref{quadratic-est}.

\end{proof}

\subsection{Nonlinear equations}

The paradifferential equation \eqref{paradiff-eqn-2} will also be used in the context of the nonlinear solutions $\varphi$. To conclude this section, we prove preliminary quadratic bounds on the inhomogeneity of the paradifferential flow, in analogy with the preceding Proposition~\ref{p:quadratic-lin} for its linearized counterpart.

\begin{proposition}\label{p:psi-eqn}
 Consider a solution $\varphi$ to \eqref{gSQG}. 
 
a) The solution $\varphi$ satisfies
\begin{equation}\label{psi-eqn}
(\D_t - c(\alpha)(T_{J^{-1}} |D_x|^{\alpha-1} - T_{\frac{1}{\alpha}(|D_x|^{\alpha-1} \D_x\psi + R)})\D_x) \varphi = f
\end{equation}
where
\begin{equation}\label{psi-eqn-est}
\||D_x|^{\frac{2-\alpha}{2}} f\|_{L^\infty} \lesssim_A B, \qquad \|\D_x f\|_{L^\infty} \lesssim_A B^2.
\end{equation}

b) The same holds for $\psi$ in the place of $\varphi$.
\end{proposition}

\begin{proof}

We first consider the case of $\varphi$. We paradifferentially decompose $Q(\varphi, \D_x \varphi)$ in \eqref{gSQG} to write it in terms of the paradifferential equation \eqref{paradiff-eqn} with source,
\[
\D_t\varphi - c(\alpha) |D_x|^{\alpha-1} \D_x \varphi - Q_{lh}(\varphi, \D_x \varphi) = Q_{hl}(\varphi, \D_x \varphi) + Q_{hh}(\varphi, \D_x \varphi).
\]

As with the linearized equation, we estimate the source terms. We see directly from definition that $Q_{hl}(\varphi, \D_x \varphi)$ has a favorable balance of derivatives which satisfies \eqref{psi-eqn-est} and may be absorbed into $f$. The balanced $Q_{hh}$ term similarly satisfies \eqref{psi-eqn-est}, so we have thus replaced $Q(\varphi, \D_x \varphi)$ in \eqref{gSQG} with $Q_{lh}(\varphi, \D_x \varphi)$. In turn, it then suffices to apply Lemma~\ref{l:q-principal} to obtain \eqref{psi-eqn}.

\

We next reduce the equation for $T_{F'(\varphi_x)} \varphi$ to that of $\varphi$. It suffices to apply the paracoefficient $T_{F'(\varphi_x)}$ to \eqref{psi-eqn}, and estimate the commutators. This is straightforward for the spatial paradifferential terms, applying Lemma~\ref{l:para-com} and observing a favorable balance of derivatives.

For the time derivative, we substitute \eqref{psi-eqn} for the time derivative of $\varphi$:
\[
T_{F''(\varphi_x) \D_x \D_t\varphi} \varphi = T_{F''(\varphi_x) \left(c(\alpha)\D_x T_{J^{-1}} |D_x|^{\alpha-1} \D_x\varphi + \frac{c(\alpha)}{\alpha}\D_x T_{|D_x|^{\alpha-1} \D_x \psi + R} \D_x \varphi + \D_x f\right)} \varphi.
\]
The estimate \eqref{psi-eqn-est} on $\D_x f$ in the paracoefficient implies that its contribution in this context also satisfies \eqref{psi-eqn-est}. For the remaining terms, the favorable balance of derivatives, with two or more derivatives on the low frequency paracoefficient, again implies that we may absorb their contribution into $f$. 

\

To conclude the proof for $\psi$, it suffices to apply the Moser estimate of Theorem~\ref{t:moser}, other than for the time derivative, for which we need to estimate 
\[
\D_x^{-1} (F'(\varphi_x) \D_x \D_t\varphi) - \D_t T_{F'(\varphi_x)} \varphi.
\]
We decompose this into
\[
[\D_x^{-1}, T_{F'(\varphi_x)}] \D_x \D_t\varphi
\]
which we estimate directly, using the favorable balance of derivatives, and 
\[
\D_x^{-1} T_{\D_x \D_t\varphi} F'(\varphi_x) + \D_x^{-1}\Pi(\D_x \D_t\varphi, F'(\varphi_x))
\]
which is similar to the time derivative commutation in the previous reduction.

\end{proof}

\section{Reduction to the paradifferential equation}\label{s:reduction}

In this section, we reduce energy estimates and well-posedness for the linearized equation \eqref{linearized-eqn} to that of the inhomogeneous paradifferential flow \eqref{paradiff-eqn},
\[
\D_tv - c(\alpha)|D_x|^{\alpha-1} \D_xv - \D_x Q_{lh}(\varphi, v) = f.
\]
For the energy estimates to be balanced, we in turn require that the inhomogeneity $f$ satisfy balanced cubic estimates.

Unfortunately, the paradifferential errors $Q_{hl}(\varphi, v)$ and $Q_{hh}(\varphi, v)$ are essentially quadratic rather than cubic, and in particular do not satisfy balanced cubic estimates. On the other hand, as we noted in \eqref{atoq}, $Q$ is the quadratic form associated to the resonance function for the dispersion relation of \eqref{gSQG}, up to leading order and a low frequency coefficient. The classical normal form transformation associated to this nonlinearity is
\begin{equation}\label{full-nf}
\tilde v = v - \frac{1}{\alpha}\D_x (\psi v), \qquad \psi = \D_x^{-1} F(\varphi_x).
\end{equation}

Unfortunately, \eqref{full-nf} has two drawbacks:
\begin{enumerate}
\item It is unbounded, and hence we cannot apply it directly, and
\item quartic (essentially cubic) error terms arising in the equation for $\tv$ given by \eqref{full-nf} are still unbalanced.
\end{enumerate}
To address the first issue, we instead consider only the bounded paradifferential components of \eqref{full-nf},
\begin{equation}\label{para-nf}
\tilde v = v - \frac{1}{\alpha}\D_x (T_v \psi) - \frac{1}{\alpha}\D_x \Pi(v, \psi),
\end{equation}
which is well-suited to the purpose of reducing the problem to the paradifferential equation \eqref{paradiff-eqn}. To address the second issue, we augment \eqref{para-nf} with a low frequency Jacobian coefficient which eliminates the quartic and higher order residuals: 
\begin{equation}\label{j-para-nf}
\tilde v = v - \frac{1}{\alpha}\D_x T_{T_{J} v} \psi - \frac{1}{\alpha}\D_x \Pi(T_{J} v, \psi), \qquad J = (1 - \D_x \psi)^{-1}.
\end{equation}

\begin{proposition}\label{Linearized normalized variable}
Consider a solution $v$ to \eqref{linearized-eqn}. Then we have
\begin{equation}\label{paradiff-eqn-inhomog}
\D_t\tv - c(\alpha)|D_x|^{\alpha-1} \D_x\tv - \D_x Q_{lh}(\varphi, \tv) = \tilde f,
\end{equation}
where $\tilde f$ satisfies balanced cubic estimates,
\begin{equation}\label{bal-est}
\|\tilde f\|_{L^2} \lesssim_A B^2 \|v\|_{L^2}.
\end{equation}
\end{proposition}

\begin{proof}

We express $v$ satisfying \eqref{linearized-eqn} in terms of the paradifferential equation \eqref{paradiff-eqn} with source,
\[
\D_tv - c(\alpha) |D_x|^{\alpha-1} \D_xv - \D_x Q_{lh}(\varphi, v) = \D_x Q_{hl}(\varphi, v) + \D_x Q_{hh}(\varphi, v).
\]
Unlike in Proposition~\ref{p:quadratic-lin}, we do not estimate the source terms directly. Instead, we will establish the following cancellation with the contribution from the normal form correction,
\begin{equation}\label{main-cancellation}
\D_t \D_x T_{T_J v} \psi - c(\alpha) |D_x|^{\alpha-1} \D_x^2 T_{T_J v} \psi - \D_x Q_{lh}(\varphi, \D_x T_{T_J v} \psi) = \alpha\D_x Q_{hl}(\varphi, v) + \tilde f,
\end{equation}
with the analogous relationship for the balanced $\Pi$ component of the correction, with $Q_{hh}$.

To show \eqref{main-cancellation}, we first observe that using Lemma~\ref{l:q-principal}, we may replace $\D_x Q_{lh}$ on the left hand side of \eqref{main-cancellation} by its principal components. The $\Gamma$ error is estimated using the second estimate of \eqref{l-est},
\[
\|\D_x \Gamma\|_{L^2} \lesssim_A B \||D|^{\frac{\alpha}{2}} \D_x T_{T_J v} \psi\|_{L^2} \lesssim_A B^2 \|v\|_{L^2}
\]
and may be absorbed into $\tilde f$. It thus suffices to show
\begin{equation}\label{lh-simplified}
\left(\D_t - c(\alpha)\D_x (T_{J^{-1}} |D_x|^{\alpha-1} - T_{\frac{1}{\alpha}(|D_x|^{\alpha-1}\D_x \psi + R)})\right) \D_x T_{T_J v} \psi = \alpha\D_x Q_{hl}(\varphi, v) + \tilde f.
\end{equation}

\

Next, we compute the time derivative in \eqref{lh-simplified}. The case where $\D_t$ falls on the low frequency $J$ may be absorbed into $\tilde f$ due to a favorable balance of derivatives. More precisely, we use \eqref{psi-eqn} to write
\[
\D_t J = J^2 \D_x \D_t \psi = c(\alpha)J^2\D_x (T_{J^{-1}} |D_x|^{\alpha-1} - T_{\frac{1}{\alpha}(|D_x|^{\alpha-1} \D_x\psi + R)})\D_x \psi + J^2 \D_x f
 \]
so that we can we can estimate for instance the contribution of the first term,
\[
\|\D_x T_{T_{ J^2 \D_x^2 |D_x|^{\alpha-1} \psi} v} \psi \|_{L^2} \lesssim_A B^2 \|v\|_{L^2},
\]
with similar estimates for the other contributions, using the first estimate of \eqref{psi-eqn-est} for the contribution of $f$.

In the remaining cases, $\D_t$ falls on the middle frequency $v$ or the high frequency $\psi$, so we use \eqref{paradiff-eqn-2} and \eqref{psi-eqn} respectively to write
\begin{equation}\label{dt-nf}
\begin{aligned}
\D_x T_{T_J \D_t v} \psi &= c(\alpha) \D_x T_{T_J\left(T_{J^{-1}} |D_x|^{\alpha-1} \D_xv + \frac{1}{\alpha}T_{|D_x|^{\alpha-1} \D_x \psi + R} \D_x v + f\right)} \psi, \\
\D_x T_{T_J v} \D_t \psi &= c(\alpha)\D_x T_{T_J v} \left(T_{J^{-1}} |D_x|^{\alpha-1} \D_x\psi + T_{\frac{1}{\alpha}(|D_x|^{\alpha-1} \D_x \psi + R)} \D_x \psi + f_\psi\right).
\end{aligned}
\end{equation}
We consider the first, second, and third contributions from the two equations of \eqref{dt-nf} in pairs:

\

i) The first terms in \eqref{dt-nf} combine with the second term on the left in \eqref{lh-simplified},
\begin{equation}\label{pre-q}
c(\alpha)\D_x (T_{T_{J} T_{J^{-1}} |D_x|^{\alpha-1} \D_xv} \psi + T_{T_{J} v} T_{J^{-1}} |D_x|^{\alpha-1} \D_x\psi - T_{J^{-1}} |D_x|^{\alpha-1} \D_x T_{T_{J} v} \psi),
\end{equation}
to form $\alpha \D_x Q_{hl}(\varphi, v)$, modulo balanced errors which may be absorbed into $\tilde f$. To see this, we will use in each of the three terms that $J^{-1}J = 1$. As we do so, we need to take care that any paraproduct errors and commutators yield balanced errors. 

First, we observe that in the third term, we can apply the commutator estimate
\[
\|\D_x [T_{J^{-1}}, |D_x|^{\alpha-1} \D_x] T_{T_{J} v} \psi \|_{L^2} \lesssim_A B^2 \|v\|_{L^2}.
\]
Then using Lemma~\ref{l:para-prod} to compose paraproducts in each of the three terms of \eqref{pre-q}, we have
\[
c(\alpha)\D_x (T_{|D_x|^{\alpha-1} \D_xv} \psi + T_{J^{-1} T_{J} v} |D_x|^{\alpha-1}\D_x\psi -  |D_x|^{\alpha-1} \D_x T_{J^{-1}T_{J} v} \psi).
\]

For the latter two terms, we will also use Lemma~\ref{l:para-prod} to compose paraproducts, before applying $J^{-1}J = 1$. To do so, we first need to exchange multiplication by $J^{-1}$ with the paraproduct $T_{J^{-1}}$. However, the error from this exchange is not directly perturbative. Instead, we perform the exchange for the two terms simultaneously, to observe a cancellation in the form of the commutator
\[
c(\alpha)\D_x(T_{T_{T_{J} v} J^{-1}} |D_x|^{\alpha-1} \D_x\psi -  |D_x|^{\alpha-1}\D_x T_{T_{T_{J} v}J^{-1}} \psi),
\]
which has a favorable balance of derivatives and may be absorbed into $\tilde f$. The same holds for the analogous cases with $\Pi(J^{-1}, T_{J} v)$. We have thus reduced \eqref{pre-q} to 
\[
c(\alpha)\D_x (T_{|D_x|^{\alpha-1} \D_xv} \psi + T_{v} |D_x|^{\alpha-1} \D_x\psi -  |D_x|^{\alpha-1} \D_x T_{v} \psi) = \alpha\D_x \Omega_{hl}(\psi, v)
\]
which by Lemma~\ref{l:toy-red} coincides with $\alpha\D_x Q_{hl}(\varphi, v)$ up to balanced errors, as desired.

\

ii) The second terms in \eqref{dt-nf},
\begin{equation}\label{second-term}
\frac{c(\alpha)}{\alpha}\D_x (T_{T_{J} T_{|D_x|^{\alpha-1} \D_x \psi + R} \D_x v} \psi + T_{T_{J} v} T_{|D_x|^{\alpha-1} \D_x \psi + R} \D_x \psi),
\end{equation}
combine to cancel the third term on the left hand side of \eqref{lh-simplified}, up to balanced errors. To see this, we apply the commutator Lemma~\ref{l:para-com} to exchange the first term of \eqref{second-term} with
\[
\frac{c(\alpha)}{\alpha}\D_x T_{T_{|D_x|^{\alpha-1} \D_x \psi + R} T_{J} \D_x v} \psi.
\]
We can freely exchange the low frequency paraproduct $T_{|D_x|^{\alpha-1} \D_x\psi + R}$ with a standard product, since
\begin{equation}\label{favorable-bal}
\|\D_x T_{T_{T_{J} \D_x v} |D_x|^{\alpha-1} \D_x \psi + T_{T_{J} \D_x v} R} \psi\|_{L^2} \lesssim_A B^2 \|v\|_{L^2}
\end{equation}
and likewise for the balanced $\Pi$ case. We thus have 
\[
\frac{c(\alpha)}{\alpha}\D_x T_{T_{J} \D_x v \cdot |D_x|^{\alpha-1} \D_x \psi + R} \psi.
\]
Then applying Lemma~\ref{l:para-prod} for splitting paraproducts, and returning to \eqref{second-term}, we arrive at
\[
\frac{c(\alpha)}{\alpha}\D_x (T_{T_{J} \D_x v} T_{|D_x|^{\alpha-1} \D_x \psi + R} \psi + T_{T_{J} v} T_{|D_x|^{\alpha-1} \D_x \psi + R} \D_x \psi).
\]
Lastly, we factor out a derivative,
\[
\frac{c(\alpha)}{\alpha}\D_x^2 T_{T_{J} v} T_{|D_x|^{\alpha-1} \D_x \psi + R} \psi
\]
where we have absorbed the cases where the factored derivative falls on $J$ or $|D_x|^{\alpha-1} \D_x \psi + R$ into $\tilde f$, similar to \eqref{favorable-bal}. After one more instance of the commutator Lemma~\ref{l:para-com}, we arrive at the third term on the left hand side of \eqref{lh-simplified} as desired.

\

iii) By Propositions~\ref{p:quadratic-lin} and \ref{p:psi-eqn} respectively, the contributions from $f$ and $f_\psi$ satisfy \eqref{bal-est} and may be absorbed into $\tilde f$.

\end{proof}

We also prove a similar but easier balanced estimate for the reduction of the nonlinear equation to the paradifferential flow, in the $\dot H^s$ setting. In this situation, the normal form correction only consists of a balanced $\Pi$ component:

\begin{equation}\label{j-para-nf-2}
\tilde{\varphi} = \varphi -  \frac{1}{\alpha}\Pi(\psi,T_J \D_x\varphi).
\end{equation}

\begin{proposition}\label{Normalized variable}
Consider a solution $\varphi$ to \eqref{SQG}. Then we have
\begin{equation}\label{paradiff-eqn-inhomog-2}
\D_t\tilde{\varphi} - 2\clogd \D_x\tilde{\varphi} -  \D_x Q_{lh}(\varphi, \tilde{\varphi}) = \tilde f,
\end{equation}
where $\tilde f$ satisfies balanced cubic estimates,
\begin{equation}\label{bal-est-2}
\|\tilde f\|_{\dot{H}^s} \lesssim_A B^2 \|\varphi\|_{\dot{H}_x^s}.
\end{equation}
\end{proposition}

\begin{proof}

First observe that we have
\[
\D_t\varphi - 2\clogd \D_x \varphi -  \D_x Q_{lh}(\varphi, \varphi) = Q_{hh}(\varphi, \D_x \varphi).
\]
Then the normal form analysis is similar to the analysis for the (balanced) paradifferential errors of the linear equation in Proposition~\ref{Linearized normalized variable}.

To see that we can obtain balanced estimates in the $\dot H^s$ setting, first observe that in each of the estimates in the proof of Proposition~\ref{Linearized normalized variable}, we can easily obtain at least one $B$ from the estimate of the low frequency variable. Then since we are in the balanced $\Pi$ setting, we can obtain a second $B$, with $s$ outstanding derivatives, which can then be placed on the remaining high frequency factor.

\end{proof}

\section{Energy estimates for the paradifferential equation}\label{s:energy estimates}

In this section we obtain energy estimates for the paradifferential flow \eqref{paradiff-eqn}. We define the modified energy
\[
E(v) := \int v \cdot T_{\tilde{J}(\psi_x)} v \, dx, \qquad \tilde{J}(x) = J(x)^{-\frac{1}{\alpha}}.
\]

Unlike in the SQG case, where $\alpha = 1$, the energy functional here is in general a fully nonlinear function of $\psi_x$. In particular, observe that the paradifferential normal form used in the previous section is a linearization of the normal form underlying the modified energy used here.

Indeed, the normal form of the previous section would suggest that here we use the transformation
\[
\displaystyle \tv=v-\frac{1}{\alpha}\partial_xT_{\psi}v,
\]
corresponding to a cubic modified energy of the form
\begin{align*}
\int \tilde{v}^2-\frac{2}{\alpha}\tilde{v}\cdot\partial_xT_{\psi}\tilde{v}\,dx=\int \tilde{v}\cdot T_{1-\frac{1}{\alpha}\psi_x}\tilde{v}\,dx,
\end{align*}
which only provides the first order approximation of the correct energy $E(v)$.

\begin{proposition}\label{Paradiferential flow linearized energy estimates}
We have
\begin{equation}\label{bal-energy}
\frac{d}{dt} E(v) \lesssim_A B^2 \|v\|_{L^2}^2 + \|f\|_{L^2} \|v\|_{L^2}.
\end{equation}

\end{proposition}


\begin{proof}
Without loss of generality we assume $f = 0$. Using \eqref{paradiff-eqn} and \eqref{psi-eqn} respectively to expand time derivatives, we have
\begin{equation}\label{dtq}
\begin{aligned}
\frac{\alpha}{c(\alpha)} \frac{d}{dt} \int v \cdot T_{\tilde J} v \, dx &= 2\frac{\alpha}{c(\alpha)} \int \D_t v \cdot T_{\tilde J} v \, dx - \frac{\alpha}{c(\alpha)} \int v \cdot T_{\frac{1}{\alpha}J \tilde J \cdot \D_x \D_t \psi} v \, dx \\
&= 2\frac{\alpha}{c(\alpha)} \int (\clogd \D_xv + \D_x Q_{lh}(\varphi, v)) \cdot T_{\tilde J} v \, dx \\
&\quad - \int v \cdot T_{ J \tilde J \cdot \D_x ( T_{J^{-1}} \logd \D_x \psi + \frac{1}{\alpha} T_{\logd \D_x \psi + R} \D_x \psi + f)} v \, dx.
\end{aligned}
\end{equation}
The contribution from $f$ may be estimated using \eqref{psi-eqn-est} and discarded. 

We expand the contribution from $Q_{lh}$ in the second term on the right hand side of \eqref{dtq} using Lemma~\ref{l:q-principal},
\begin{equation}\label{quartic-terms-2}
2 \int \D_x (T_{\logd \D_x \psi + R} v + T_{\psi}\logd \D_x  v - \logd \D_x T_\psi v) \cdot T_{\tilde J} v \, dx.
\end{equation}
For clarity we also record the remaining terms on the right hand side of \eqref{dtq},
\begin{equation}\label{quartic-terms-3}
\int 2\alpha \logd \D_xv \cdot T_{\tilde J} v \, dx - \int v \cdot T_{J \tilde J \cdot \D_x ( T_{J^{-1}} \logd \D_x \psi + \frac{1}{\alpha} T_{\logd \D_x \psi + R} \D_x \psi)} v \, dx.
\end{equation}
We observe cancellations between \eqref{quartic-terms-2} and \eqref{quartic-terms-3} in the following three steps.

\

i) From the first term in \eqref{quartic-terms-2}, we have after cyclically integrating by parts, 
\[
\int v \cdot T_{\D_x^2 \logd \psi + \D_x R}T_{\tilde J} v - v \cdot T_{\logd \D_x  \psi + R}T_{\D_x \tilde J} v + v \cdot [T_{\tilde J}, T_{\logd \D_x  \psi + R}]\D_x v \, dx.
\]
The commutator satisfies \eqref{bal-energy} by Lemma~\ref{l:para-com}. On the other hand, the first two terms cancel with the second integral in \eqref{quartic-terms-3}, up to errors also satisfying \eqref{bal-energy}. To see these cancellations, we first observe that 
\[
JJ^{-1} = 1, \qquad \D_x \tilde J =-\frac{1}{\alpha} J \tilde J \D_x^2 \psi.
\]
To make use of these identities, we in turn use Lemma~\ref{l:para-prod} to compose paraproducts, and observe that any contribution with two or more derivatives on the lowest frequency has a favorable balance of derivatives and satisfies \eqref{bal-energy}. For instance, from the second term of the second integrand in \eqref{quartic-terms-3}, we have the perturbative component
\[
\left|\int v \cdot T_{T_{\D_x^2\logd \psi} \D_x\psi} v \, dx \right| \lesssim_A B^2 \|v\|_{L^2}^2.
\]

\

ii) We consider the remaining two terms of \eqref{quartic-terms-2}, which may be viewed as a (skew symmetric) commutator, plus a residual integral:
\begin{equation}\label{ii2}
2 \int \D_x [T_{\psi},\logd] \D_x v \cdot T_{\tilde J} v \, dx - 2 \int \D_x  \logd T_{\D_x\psi} v \cdot T_{\tilde J} v \, dx.
\end{equation}
We would like to replace the commutator with its principal component,
\[
-2(\alpha - 1) \int T_{\D_x \psi} \logd \D_x v \cdot T_{\tilde J} v \, dx.
\]
However, since this principal component is not directly skew symmetric, we also include an additional correction:
\[
(\alpha - 1) \int [T_{\D_x \psi}, \logd \D_x] v \cdot T_{\tilde J} v \, dx.
\]
After peeling away the principal component and the correction from the commutator, the residual is skew-symmetric:
\[
L(\D_x^2 \psi, \logd (\cdot)) = \D_x [T_{\psi},\logd] \D_x + (\alpha - 1)(T_{\D_x \psi} \logd \D_x - \half [T_{\D_x \psi}, \logd \D_x]).
\]
We thus have a commutator form for its contribution,
\[
\int T_{\tilde J} L(\D_x^2 \psi, \logd v) \cdot v \, dx = \half \int (T_{\tilde J} L(\D_x^2 \psi, \logd v) -  L(\D_x^2 \psi, \logd T_{\tilde J} v)) \cdot v \, dx
\]
for which we have the desired balanced estimate. Collecting the principal component, the correction, and the second integral from \eqref{ii2}, we have
\[
 - \int \left((\alpha + 1) \logd \D_x T_{\D_x\psi} v + (\alpha - 1) T_{\D_x \psi} \logd \D_x v \right) \cdot T_{\tilde J} v \, dx.
\]
Together with the first term of \eqref{quartic-terms-3}, it remains to consider
\begin{equation}\label{final-com}
\int \left((\alpha + 1)\logd \D_x T_{J^{-1}} v + (\alpha - 1) T_{J^{-1}} \logd \D_x v \right) \cdot T_{\tilde J} v \, dx.
\end{equation}

\

iii) From the unit-coefficient terms in \eqref{final-com}, we have the commutator
\[
\int [\logd \D_x, T_{J^{-1}}] v \cdot T_{\tilde J} v \, dx =: \int L(\D_x^2 \psi, \logd v) \cdot T_{\tilde J} v \, dx.
\]
On the other hand, from the $\alpha$-coefficient terms, we integrate by parts on the first term, and commute the paracoefficients on the second term using Lemma~\ref{l:para-com},
\[
\alpha \int - T_{J^{-1}} v \cdot \logd \D_x T_{\tilde J} v + T_{\tilde J}\logd \D_x v \cdot T_{J^{-1}} v \, dx,
\]
thus also obtaining a commutator,
\[
-\alpha \int [\logd \D_x, T_{\tilde J}] v \cdot T_{J^{-1}} v \, dx =\int L(\alpha\D_x \tilde J, \logd v) \cdot T_{J^{-1}} v \, dx.
\]
Then since
\[
\alpha \D_x \tilde J =-J \tilde J \D_x^2 \psi,
\]
we may apply Lemma~\ref{l:para-prod3} to reduce to
\[
-\int T_{J \tilde J} L(\D_x^2\psi, \logd v) \cdot T_{J^{-1}} v \, dx.
\]
Using self-adjointness with Lemma~\ref{l:para-prod} and $JJ^{-1} = 1$, we have cancellation with the previous commutator. Precisely, to split and recombine the $J$ and $\tilde J$ terms while maintaining balanced errors, derivatives are shifted from the low frequency $\D_x^2 \psi$ to $v$, before being absorbed using Lemma~\ref{l:para-prod} on the $J$ and $\tilde J$ paraproducts.

\end{proof}


\begin{proposition}\label{Linearized equation energy estimates}
Assume that $A\ll 1$ and $B\in L_t^2$. Then there exists an energy functional $E_{lin}(v)$ such that we have the following:
\begin{enumerate}
    \item[a)]Norm equivalence:
    \begin{align*}
        E_{lin}(v)\approx_A\|v\|^2_{L_x^2}
    \end{align*}
    \item[b)]Energy estimates:
    \begin{align*}
        \frac{d}{dt}E_{lin}(v)\lesssim _{A}B^2\|v\|^2_{L_x^2}
    \end{align*}
\end{enumerate}
\end{proposition}
\begin{remark}
We note that the linearized equation \eqref{linearized-eqn} is well-posed in $L_x^2$. We do not need this result in the current paper, but we briefly discuss the main steps in its proof. The main idea is to prove a similar estimate for the adjoint equation, interpreted as a backward evolution in the space $L^2$. More precisely, the adjoint equation corresponding to the linearized equation has the form
\begin{align*}
    \partial_tv-c(\alpha)|D_x|^{\alpha-1}\partial_xv-Q(\varphi,\partial_xv)=0.
\end{align*}
By carrying out a paradifferential normal form transformation, similar to the one from the proof of Proposition \ref{Linearized normalized variable}, this can be reduced to
\begin{align*}
    \partial_tv-c(\alpha)|D_x|^{\alpha-1}\partial_xv-Q_{lh}(\varphi,\partial_xv)=0.
\end{align*}
By considering the modified energy functional
\begin{align*}
  \int v\cdot T_{J^{\frac{1}{\alpha}}}v\,dx, 
\end{align*}
 we obtain the desired energy estimate for the dual problem in a manner similar to Proposition \ref{Paradiferential flow linearized energy estimates}. We can now infer existence for the solutions to the linearized equation \eqref{linearized-eqn} by a standard duality argument (for the general theory, see Theorem 23.1.2 in \cite{Hormander}).
\end{remark}
\begin{proof}
Let $E_{lin}(v)=E(\tilde{v})$, where $E(\cdot)$ is defined in Proposition \ref{Paradiferential flow linearized energy estimates}, and $\tilde{v}$ is defined in Proposition \ref{Linearized normalized variable}. Part a) is immediate, whereas part b) follows from Proposition \ref{Paradiferential flow linearized energy estimates}.
\end{proof}

\section{Higher order energy estimates}\label{s:higher order energy estimates}

In this section we obtain higher order energy estimates. As the commutators with $D^s$ are quadratic and thus not perturbative, they pose an additional obstacle. To address this, we first eliminate the highest order terms by using a Jacobian exponential conjugation, at the cost of generating additional terms, but at lower order. We then carry out a normal form analysis in order to eliminate the remaining non-perturbative terms.

\begin{proposition}\label{High energy normalized variable}
    Let $s\geq 0$. Given $v$ solving \eqref{paradiff-eqn-inhomog}, there exists a normalized variable $v^s$ such that
    \begin{align*}
        \partial_tv^s - c(\alpha)|D_x|^{\alpha-1} \partial_xv^s-\partial_xQ_{lh}(\varphi,v^s)= f + \mathcal{R}(\varphi, v),
    \end{align*}
    with
    \begin{align*}
        \|v^s-|D_x|^sv\|_{L_x^2}&\lesssim_AA\|v\|_{\dot{H}_x^s}
    \end{align*}
    and $\mathcal{R}(\varphi,v)$ satisfying balanced cubic estimates,
\begin{equation}\label{bal-est-eng}
\|\mathcal{R}(\varphi, v)\|_{L^2} \lesssim_A B^2 \|v\|_{L^2}.
\end{equation}
\end{proposition}

\begin{proof}
Let $v$ satisfy \eqref{paradiff-eqn-inhomog}, where without loss of generality, $f = 0$. The natural approach is to reduce the equation for $v^s := |D_x|^s v$ to \eqref{paradiff-eqn-inhomog} with a perturbative inhomogeneity. However, the commutators arising from such a reduction are quadratic, and cannot satisfy balanced cubic estimates. In particular, they cannot be seen as directly perturbative. We will address these errors via a conjugation combined with a normal form correction.

In preparation, we use Lemmas~\ref{l:toy-red} and \ref{l:q-principal} to rewrite $Q_{lh}$ in \eqref{paradiff-eqn-inhomog}. Denoting
\[
\mathcal P = T_{1 - \frac{1}{\alpha} \D_x \psi}|D_x|^{\alpha-1}+ T_{\frac{1}{\alpha}(|D_x|^{\alpha-1} \D_x \psi + R)} + \D_x [T_{\frac{1}{\alpha}\psi}, |D_x|^{\alpha-1}],
\]
we obtain
\begin{equation}
\begin{aligned}
\D_t v - c(\alpha) \D_x \mathcal P v = 0.
\end{aligned}
\end{equation}
Then $v^s := |D_x|^s v$ satisfies the same equation, but with the following commutators in the source:
\begin{equation}\label{vs-eqn}
\begin{aligned}
&\frac{c(\alpha)}{\alpha} \D_x ([|D_x|^s, T_{-\D_x \psi}]|D_x|^{\alpha-1} + [|D_x|^s, T_{|D_x|^{\alpha-1} \D_x \psi + R}]  + \partial_x[|D_x|^s, [T_{\psi}, |D_x|^{\alpha-1}]] ) v \\
&\quad =: L(\D_x^2 \psi, |D_x|^{\alpha-1} v^s) - L(|D_x|^{\alpha-1} \D_x^2 \psi, v^s)+ c(\alpha)\partial_x^2[|D_x|^s, [T_{\frac{1}{\alpha}\psi}, |D_x|^{\alpha-1}]]v+ \mathcal R
\end{aligned}
\end{equation}
where we have absorbed $\D_x R$ into $\mathcal R$ and $L$ denotes an order zero paradifferential bilinear form, 
\begin{equation}
\begin{aligned}
L(\D_x f, u) &= -\frac{c(\alpha)}{\alpha} \D_x [|D_x|^s, T_f] |D_x|^{-s} u.
\end{aligned}
\end{equation}
In particular, observe that the principal term of $L$ is given by
\[
L(g, u) \approx -\frac{c(\alpha)}{\alpha} sT_{g} u.
\]

To address the two $L$ contributions, which are quadratic and not directly perturbative, we apply two steps:
\begin{enumerate}
\item[a)] We first apply a conjugation to $v^s$ which improves the leading order of the contributions of the $L$ terms from $\alpha-1$ and 0 to $\alpha-2$ and $-1$ respectively.
\item[b)] We then apply a normal form transformation yielding cubic, balanced source terms.
\end{enumerate}

\

a) We begin by computing the equation for the conjugated variable
\[
\tv^s := T_{J^{-\fracsa}} v^s.
\]
To do so, it suffices to apply $T_{J^{-\frac{s}{\alpha}}}$ to \eqref{vs-eqn} and consider the commutators. These will include a $\D_t$ commutator, a $\D_x$ commutator, and a $|D_x|^{\alpha-1}$ commutator.

\

i) First, we use \eqref{psi-eqn} to expand the $\D_t$ commutator,
\begin{equation*}
\begin{aligned}
\, [T_{J^{-\fracsa}}, \D_t] v^s &= - \frac{s}{\alpha}T_{J^{1-\fracsa} \D_x \D_t \psi} v^s \\
&= - \frac{s}{\alpha}T_{J^{1 - \frac{s}{\alpha}}\D_x \left(c(\alpha)T_{J^{-1}} |D_x|^{\alpha-1} \D_x \psi + \frac{c(\alpha)}{\alpha}T_{|D_x|^{\alpha-1} \D_x \psi + R} \D_x \psi + f\right)} v^s.
\end{aligned}
\end{equation*}
Here the contribution from $f$ may be estimated using \eqref{psi-eqn-est} and discarded. Further, due to a favorable balance of derivatives when $\D_x$ falls on the lowest frequency variables, we reduce to 
\[
- \frac{sc(\alpha)}{\alpha}T_{J^{-\fracsa} \left(T_{J} T_{J^{-1}} |D_x|^{\alpha-1} \D_x^2 \psi + \frac{1}{\alpha}T_{J} T_{|D_x|^{\alpha-1} \D_x \psi + R} \D_x^2 \psi\right)} v^s.
\]
Lastly, we apply Lemma~\ref{l:para-prod} to merge and split paraproducts, reducing to 
\begin{equation}\label{dt-comm}
- \frac{sc(\alpha)}{\alpha}T_{|D_x|^{\alpha-1} \D_x^2 \psi +\frac{1}{\alpha} T_{J|D_x|^{\alpha-1} \D_x \psi + R} \D_x^2 \psi} \tv^s.
\end{equation}
Observe that the first part of \eqref{dt-comm} cancels with the principal term of the second $L$ on the right hand side of \eqref{vs-eqn}. The second part of \eqref{dt-comm} will cancel with a contribution from ii) below.

\

ii) Next, we consider the commutator of $T_{J^{-\fracsa}}$ with the outer $\D_x$. We obtain 
\[
\frac{sc(\alpha)}{\alpha}T_{J^{1-\fracsa} \D_x^2 \psi} \mathcal P v^s.
\]
Here it is convenient to apply \eqref{qlh-expression-full} of Lemma~\ref{l:q-principal} to write this as
\[
\frac{sc(\alpha)}{\alpha}T_{J^{1-\fracsa} \D_x^2 \psi} \left((T_{J^{-1}} |D_x|^{\alpha-1} + T_{\frac{1}{\alpha}(|D_x|^{\alpha-1} \D_x \psi + R)}) v^s + \Gamma(\D_x^2 \psi, |D_x|^{\alpha-1}\D_x^{-1} v^s)\right).
\]
We first apply Lemma~\ref{l:para-prod3} to split $J^{1 - \fracsa}$ from $\D_x^2 \psi$ in the outer paraproduct. Then applying Lemma \ref{l:para-prod} along with \ref{l:para-com}, to compose and commute paraproducts, this reduces, modulo perturbative terms, to 
\[
\frac{sc(\alpha)}{\alpha}T_{\D_x^2 \psi} |D_x|^{\alpha-1} \tv^s + \frac{sc(\alpha)}{\alpha^2} T_{J|D_x|^{\alpha-1}\D_x \psi \cdot \D_x^2 \psi + \D_x^2 \psi \cdot R} \tv^s.
\]
The first term above cancels with the principal term of the first $L$. The second term cancels with the remaining part of the $\D_t$ commutator above in \eqref{dt-comm}. To see this cancellation, we have freely exchanged multiplication by $J|D_x|^{\alpha-1} \D_x \psi$ with a paraproduct, as the difference has a favorable balance of derivatives and is thus perturbative.

\

iii) Returning to the commutator of $T_{J^{-\fracsa}}$ with the dispersive term, it remains to consider the commutator with the inner $|D_x|^{\alpha-1}$, where we have used Lemma~\ref{l:para-com} to discard any paraproduct commutators. We have
\[
-c(\alpha)\D_x T_{J^{-1}} [T_{J^{-\fracsa}}, |D_x|^{\alpha-1}] v^s
\]
whose principal term $\fracsa c(\alpha)(\alpha-1)T_{\D_x^2 \psi} |D_x|^{\alpha-1}\tv^s$ cancels with the principal term of the double commutator on the right hand side of \eqref{vs-eqn}.

\

To conclude, we have
\begin{equation}\label{tvs-eqn}
\begin{aligned}
\D_t \tv^s &- c(\alpha) \D_x \mathcal P \tv^s \\
&= (L(\D_x^2 \psi, |D_x|^{\alpha-1} \tv^s) + \frac{sc(\alpha)}{\alpha}T_{\D_x^2 \psi} |D_x|^{\alpha-1}\tv^s) \\
&\quad - (L(|D_x|^{\alpha-1} \D_x^2 \psi, \tv^s) + \frac{sc(\alpha)}{\alpha}T_{|D_x|^{\alpha-1} \D_x^2 \psi} \tv^s) \\
&\quad + \frac{c(\alpha)}{\alpha}T_{J^{-\fracsa}}\D_x^2 [|D_x|^s, [T_\psi, |D_x|^{\alpha-1}]] v - c(\alpha)\D_x T_{J^{-1}} [T_{J^{-\fracsa}}, |D_x|^{\alpha-1}] v^s + f \\
&=: L_0(\D_x^3 \psi,  |D_x|^{\alpha-1}\D_x^{-1}\tv^s) - L_0(|D_x|^{\alpha-1} \D_x^3 \psi, \D_x^{-1}\tv^s) + L_1(\D_x^3 \psi, |D_x|^{\alpha-1}\D_x^{-1} \tv^s) + f
\end{aligned}
\end{equation}
where $f$ satisfies \eqref{bal-est-eng}. Here $L_0$ and $L_1$ denote order zero paradifferential bilinear forms, respectively
\begin{equation*}
\begin{aligned}
L_0(\D_x^2 f, |D_x|^{\alpha-1}\D_x^{-1} u) &= L(\D_x f, u) + \frac{sc(\alpha)}{\alpha}T_{\D_x f} u, \\
L_1(\D_x^2 f, |D_x|^{\alpha-1}\D_x^{-1} T_{J^{-\fracsa}} u) &=  \frac{c(\alpha)}{\alpha}T_{J^{-\fracsa}}\D_x^2 [|D_x|^s, [T_{\D_x^{-1} f}, |D_x|^{\alpha-1}] |D_x|^{-s} u \\
&\quad - c(\alpha)\D_x T_{J^{-1}} [T_{J^{-\fracsa}}, |D_x|^{\alpha-1}] u.
\end{aligned}
\end{equation*}
Observe that since $L_i$ are all order 0 paradifferential bilinear forms, we have reduced the terms of the inhomogeneity to order $-1$.

\

b) We next choose a normal form transformation to reduce the quadratic components of the inhomogeneity to balanced cubic terms. Let
\[
\tw^s = \frac{1}{c(\alpha)\alpha} T_{J} L_0(\D_x^2\psi, \D_x^{-1} \tv^s).
\]
Then we claim that $\tilde u^s := \tv^s - \tilde w^s$ is the desired normal form transform. To see this, it remains to compute $(\D_t - c(\alpha) \D_x \mathcal P) \tw^s$, which may be expressed using Lemma~\ref{l:q-principal} as
\begin{equation}\label{compute-nf}
\begin{aligned}
\left(\D_t - c(\alpha) \D_x (T_{\frac{1}{\alpha}(|D_x|^{\alpha-1}  \D_x \psi + R)} - T_{J^{-1}} |D_x|^{\alpha-1}) \right) \tw^s + \D_x\Gamma(\D_x^2 \psi, |D_x|^{\alpha-2} \tw^s),
\end{aligned}
\end{equation}
and observe cancellation with the three $L_i$ bilinear forms on the right hand side of \eqref{tvs-eqn}. To see this, we partition the computation into the following subgroups:

\

i) When the full equation of \eqref{compute-nf} falls on the high frequency $\tv^s$ input of $L_0$, we may use \eqref{tvs-eqn} to see that the contribution has a favorable balance of derivatives and may be absorbed into $f$. 

\

ii) We may commute the equation freely with the low frequency $J$ due to a favorable balance of derivatives, absorbing the contribution again into $f$. 

\

It remains to consider commutators of the terms of the equation \eqref{compute-nf} across the low frequency $\D_x^2 \psi$ input of $\tilde w^s$. 

\

iii) We first consider the commutators involving the operators
\[
c(\alpha) T_{\frac{1}{\alpha}(|D_x|^{\alpha-1} \D_x \psi + R)} \D_x + \Gamma(\D_x^2 \psi, |D_x|^{\alpha-1}\D_x^{-1} (\cdot)) \circ \D_x.
\] 
of \eqref{compute-nf}, where we have freely commuted the $\D_x$ forward with balanced errors. The contribution from the $\Gamma$ term may also be absorbed into $f$ due to a favorable balance. The remaining contribution
\[
\frac{1}{\alpha^2}T_{J} L_0(T_{|D_x|^{\alpha-1} \D_x \psi + R}\D_x^3\psi, \D_x^{-1} \tv^s)
\]
will cancel with a contribution of step iv) below.

\

iv) For the case when $\D_t$ falls on the low frequency input of $L_0$, we apply equation \eqref{psi-eqn}. Precisely, the two non-perturbative contributions on the left hand side of \eqref{psi-eqn} cancel respectively with the second $L_0$ source term in \eqref{tvs-eqn}, and the remaining contribution of step iii) above.

\

v) From the remaining dispersive term $c(\alpha)T_{J^{-1}}\D_x |D_x|^{\alpha-1}$ of \eqref{compute-nf}, the case when $\D_x$ falls on the low frequency input of $L_0$ while $|D_x|^{\alpha-1}$ has commuted to the high frequency input cancels with the first $L_0$ source term in \eqref{tvs-eqn}.

\

vi) From the same term $c(\alpha)T_{J^{-1}}\D_x |D_x|^{\alpha-1}$, it remains to consider the commutators with $|D_x|^{\alpha-1}$, where the $\D_x$ remains in front. 

We first apply Lemma~\ref{l:para-prod3} to split $J^{1 - \fracsa}$ from $\D_x^2 \psi$ in the outer paraproduct, together with Lemma~\ref{l:para-prod} with $JJ^{-1} = 1$, and opening the definition of $L_0$, we have
\[
\frac{c(\alpha)}{\alpha}\D_x^2 [|D_x|^{\alpha-1}, [|D_x|^s, T_\psi]] |D_x|^{-s} \tv^s - \frac{sc(\alpha)}{\alpha}\D_x [|D_x|^{\alpha-1}, T_{\D_x \psi}] \tv^s.
\]

We claim that these two terms cancel with the two terms of $L_1$ respectively. Indeed, for the first term, we commute using Lemma~\ref{l:para-com} to reduce to (suppressing a factor of $c(\alpha)$ from all terms henceforth)
\[
\frac{1}{\alpha}T_{J^{-\fracsa}}\D_x^2 [|D_x|^{\alpha-1}, [|D_x|^s, T_\psi]] v
\]
which cancels with the double commutator term of $L_1$. For the second term, we also commute using Lemma~\ref{l:para-com} to reduce to
\[
\frac{s}{\alpha} T_{J^{-\fracsa}} \D_x [T_{\D_x \psi}, |D_x|^{\alpha-1}] v^s =: - \frac{s(\alpha-1)}{\alpha} T_{J^{-\fracsa}} T_{\D_x^2 \psi} |D_x|^{\alpha-1}v^s + T_{J^{-\fracsa}} L_2(\D_x^3 \psi, |D_x|^{\alpha-1}\D_x^{-1} v).
\]
On the other hand, the second term of $L_1$ may be expressed as
\[
- \D_x T_{J^{-1}} [T_{J^{-\fracsa}}, |D_x|^{\alpha-1}] v^s = \frac{s(\alpha-1)}{\alpha} T_{J^{-1}}T_{J^{1-\fracsa} \D_x^2 \psi} v^s - T_{J^{-1}} L_2(J^{1-\fracsa} \D_x^3 \psi, |D_x|^{\alpha-1}\D_x^{-1} v).
\]
We apply Lemma~\ref{l:para-prod3} to split $J^{1 - \fracsa}$ into its own paraproduct. Then these terms cancel, up to an application of Lemma~\ref{l:para-prod}. Precisely, we move one derivative from $\partial_x^2\psi$ onto the highest frequency, after which we apply Lemma~\ref{l:para-prod} with $JJ^{-1} = 1$.
\end{proof}

\

We thus obtain the following:

\begin{proposition}\label{Higher energy estimates}
Assume that $A\ll 1$ and $B\in L_t^2$. Let $s\geq 0$. Then there exist energy functionals $E^{(s)}(v)$ such that we have the following:
\begin{enumerate}
    \item[a)]Norm equivalence:
    \begin{align*}
        E^{(s)}(v)\approx_A\|v\|^2_{\dot{H}_x^s}
    \end{align*}
    \item[b)]Energy estimates:
    \begin{align*}
        \frac{d}{dt}E^{(s)}(v)\lesssim _{A}B^2\|v\|^2_{\dot{H}_x^s}
    \end{align*}
\end{enumerate}
\end{proposition}
\begin{proof}
Let $E^{(s)}(v)=E(v^s)$, where $E(\cdot)$ is defined in Proposition \ref{Paradiferential flow linearized energy estimates}, and $v^s$ is defined in Proposition \ref{High energy normalized variable}. Part a) is immediate, whereas part b) follows from Proposition \ref{Paradiferential flow linearized energy estimates}.
\end{proof}

\section{Local well-posedness}\label{s:lwp}

In this section we prove Theorem \ref{t:lwp}, which is our main local well-posedness result. We follow the approach outlined in \cite{ITprimer}. We consider $\varphi_0 \in (\dot{H}_x^{s_1}\cap\dot{H}_x^{s_2})$, with $s_1<\frac{3}{2}$, $s_2>\frac{\alpha+3}{2}$. Let $\varphi_0^h=(\varphi_0)_{\leq h}$, where $h\in\mathbb{Z}$. Since $\varphi_0^h \rightarrow \varphi_0$ in $(\dot{H}_x^{s_1}\cap\dot{H}_x^{s_2})$, we may assume that $ \|\varphi_0^h\|_{(\dot{H}_x^{s_1}\cap\dot{H}_x^{s_2})}<R$ for all $h$.

We construct a uniform $(\dot{H}_x^{s_1}\cap\dot{H}_x^{s_2}) $ frequency envelope $\{c_k\}_{k\in\mathbb{Z}}$ for $\varphi_0$ having the following properties:

\begin{enumerate}
     \item[a)]Uniform bounds:     
     \[ \|P_k(\varphi_0^h)\|_{\dot {H}_x^{s_1}\cap\dot {H}_x^{s_2}}\lesssim c_k,\]
     
     \item[b)]High frequency bounds:     
     \[\|\varphi_0^h\|_{\dot {H}_x^{s_1}\cap\dot {H}_x^{N}}\lesssim 2^{h(N-s_2)}c_h, \qquad N>s_2, \]
     
     \item[c)]Difference bounds:     
     \[\|\varphi_0^{h+1}-\varphi_0^h\|_{L_x^2}\lesssim 2^{-s_2h}c_h,\]
     
     \item[d)]Limit as $h\rightarrow\infty$:     
     \[ \varphi_0^h\rightarrow \varphi_0 \in \dot{H}_x^{s_1}\cap\dot{H}_x^{s_2}.\]
     
 \end{enumerate}

Let $\varphi^h$ be the solutions with initial data $\varphi_0^h$, whose existence is guaranteed instance by \cite{SQGzero}. Using the energy estimate for the solution $\varphi$ of \eqref{gSQG} from Proposition \ref{Higher energy estimates} and Proposition \ref{Normalized variable}, we deduce that there exists $T = T(\|\varphi_0\|_{H_x^s}) > 0$ on which all of these solutions are defined, with high frequency bounds
    \[ 
    \|\varphi^h\|_{C_t^0(\dot{H}_x^{s_1}\cap\dot{H}_x^{N})}\lesssim \|\varphi_0^h\|_{(\dot{H}_x^{s_1}\cap\dot{H}_x^{N})} \lesssim 2^{h(N-s_2)}c_h.
     \]
Further, by using the energy estimates for the solution of the linearized equation from Proposition \ref{Linearized equation energy estimates}, we have
\[
\|\varphi^{h+1}-\varphi^h\|_{C_t^0L_x^2}\lesssim 2^{-s_2h}c_h.
\]
By interpolation, we infer that
\[
\|\varphi^{h+1}-\varphi^h\|_{C_t^0(\dot{H}_x^{s_1}\cap\dot{H}_x^{s_2})}\lesssim c_h.
\]

As in \cite{ITprimer}, we get
\[
\|P_k \varphi^h \|_{C_t^0(\dot{H}_x^{s_1}\cap\dot{H}_x^{s_2}) } \lesssim c_k
\]
and that
\[
\|\varphi^{h+k}-\varphi^h\|_{C_t^0(\dot{H}_x^{s_1}\cap\dot{H}_x^{s_2}) }\lesssim c_{h\leq\cdot<h+k}=\left(\sum_{\substack{n=h}}^{h+k-1}c_n^2\right)^{\frac{1}{2}}
\]
for every $k\geq 1$. Thus, $\varphi^h$ converges to an element $\varphi$ belonging to $C_t^0(\dot{H}_x^{s_1}\cap\dot{H}_x^{s_2})([0,T]\times\mathbb{R})$.  Moreover, we also obtain
\begin{equation}\label{convergence estimate}
\begin{aligned}
\|\varphi^h - \varphi\|_{C_t^0(\dot{H}_x^{s_1}\cap\dot{H}_x^{s_2})} &\lesssim c_{\geq h}=\left(\sum_{\substack{n=h}}^{\infty} c_n^2\right)^{\frac{1}{2}}.
\end{aligned}
\end{equation}

We now prove continuity with respect to the initial data. We consider a sequence
 \[
 \varphi_{0j}\rightarrow \varphi_0 \in \dot{H}_x^{s_1}\cap\dot{H}_x^{s_2}
 \]
 and an associated sequence of $ H_x^s$-frequency envelopes $\{c^j_k\}_{k \in \Z}$, each satisfying the analogous properties enumerated above for $c_k$, and further such that $c^j_k \rightarrow c_k$ in $l^2(\mathbb{Z})$. In particular,
\begin{equation}\label{convergence estimate j}
\|\varphi_j^h - \varphi_j\|_{C_t^0(\dot{H}_x^{s_1}\cap\dot{H}_x^{s_2})} \lesssim c^j_{\geq h}=\left(\sum_{\substack{n=h}}^{\infty} (c^j_n)^2\right)^{\frac{1}{2}}.
\end{equation}
 
Using the triangle inequality with \eqref{convergence estimate} and \eqref{convergence estimate j}, we write
\begin{align*}
\|\varphi_j - \varphi\|_{C_t^0(\dot{H}_x^{s_1}\cap\dot{H}_x^{s_2}) } &\lesssim \|\varphi^h - \varphi\|_{C_t^0(\dot{H}_x^{s_1}\cap\dot{H}_x^{s_2})}+\|\varphi_j^h - \varphi_j\|_{C_t^0(\dot{H}_x^{s_1}\cap\dot{H}_x^{s_2})}+\|\varphi_j^h - \varphi^h\|_{C_t^0(\dot{H}_x^{s_1}\cap\dot{H}_x^{s_2})}\\
&\lesssim c_{\geq h}+c^j_{\geq h}+\|\varphi_j^h - \varphi^h\|_{C_t^0(\dot{H}_x^{s_1}\cap\dot{H}_x^{s_2})}.
\end{align*}
To address the third term, we observe that for every fixed $h$, $\varphi_j^h \rightarrow \varphi^h$ in $\dot{H}_x^{s_1}\cap\dot{H}_x^{s_2}$. We conclude $\varphi_j \rightarrow \varphi$ in $C_t^0(\dot{H}_x^{s_1}\cap\dot{H}_x^{s_2}) ([0,T]\times\mathbb{R})$ and therefore $\varphi_j\rightarrow \varphi$ in $C_t^0 (\dot{H}_x^{s_1}\cap\dot{H}_x^{s_2})([0,T]\times\mathbb{R})$.

\section{Global well-posedness}\label{s:gwp}

In this section we prove the global well-posedness part of Theorem \ref{t:gwp} and the dispersive bounds of the resulting solution, by using the wave packet testing method of Ifrim-Tataru. This approach is systematically presented in \cite{ITpax}.

\subsection{Notation}

Consider the linear flow
\[
i\D_t \varphi - A(D)\varphi = 0
\]
and the linear operator
\[
L = x - tA'(D).
\]
In our setting, we have the symbol
\[
a(\xi) = -c(\alpha)\xi |\xi|^{\alpha-1}
\]
and thus
\[
A(D)=-c(\alpha)D|D|^{\alpha-1}, \qquad  L = x+t\alpha c(\alpha)|D|^{\alpha-1}.
\]
We define the weighted energy space ($s_0<1$, $s>\alpha+2$)
\[
\|\varphi\|_X = \|\varphi\|_{\dot H^{s_0}\cap \dot{H}^s} + \|L\partial_x \varphi\|_{L^2}.
\]

\
We also define the pointwise control norm
\[
\|\varphi\|_Y = \||D_x|^{1-\delta}\varphi\|_{L_x^\infty} + \||D_x|^{\frac{\alpha}{2}+\delta}\varphi_x\|_{L_x^\infty}.
\]

We partition the frequency space into dyadic intervals $I_\lambda$ localized at dyadic frequencies $\lambda \in 2^\Z$, and consider the associated partition of velocities
\[
J_\lambda = a'(I_\lambda)
\]
which form a covering of $(-\infty,0)$. For each $\lambda$, $|J_\lambda|\approx \lambda^{\alpha-1}$. To these intervals $J_\lambda$ we select reference points $v_\lambda \in J_\lambda$, and consider an associated spatial partition of unity
\[
1 = \sum_\lambda \chi_\lambda(x), \qquad \text{supp } \chi_\lambda \subseteq  \overline{J_\lambda},\qquad \chi_\lambda=1\text{ on }J_\lambda,
\]
where $\overline{J_\lambda}$ is a slight enlargement of $J_\lambda$, of comparable length, uniformly in $\lambda$.

Lastly, we consider the related spatial intervals, $tJ_\lambda$, with reference points $x_\lambda  = tv_\lambda \in tJ_\lambda$. 
\

\subsection{Overview of the proof}

We provide a brief overview of the proof.

\

1. We make the bootstrap assumption for the pointwise bound
\begin{equation}\label{bootstrap}
    \|\varphi(t)\|_Y \lesssim C \eps \langle t\rangle^{-\frac{1}{2}}
\end{equation}
where $C$ is a large constant, in a time interval $t \in [0, T]$ where $T>1$. 

\

2. The energy estimates for \eqref{gSQG} and the linearized equation will imply
\begin{equation}\label{Regular Energy Estimate}
    \|\varphi(t)\|_{X} \lesssim \langle t\rangle^{C^2\eps^2} \|\varphi(0)\|_X.
    \end{equation}
    
\

3. We aim to improve the bootstrap estimate \eqref{bootstrap} to 
\begin{equation}\label{pointwise}
    \|\varphi(t)\|_Y \lesssim \eps \langle t\rangle^{-\frac12}.
\end{equation}
We use vector field inequalities to derive bounds of the form
\begin{equation}\label{pt-estimate}
    \|\varphi(t)\|_Y\lesssim \eps \langle t\rangle^{-\frac{1}{2} + C\eps^2},
\end{equation}
which is the desired bound but with an extra $t^{C\eps^2}$ loss.

\

4. In order to rectify the extra loss, we use the wave packet testing method. Namely, we define a suitable asymptotic profile $\gamma$, which is then shown to be an approximate solution for an ordinary differential equation. This enables us to obtain suitable bounds for the asymptotic profile without the aforementioned loss, which can then be transferred back to the solution $\varphi$.

\subsection{Energy estimates} 

From Proposition \ref{Higher energy estimates} and Gr\"onwall's lemma, together with the fact that $\eps\ll 1$, we get that
\begin{align*}
 \|\varphi(t,x)\|_{H_x^s}\lesssim e^{C\int_0^tC(A(\tau))B(\tau)^2\,d\tau}\|\varphi_0\|_{H_x^s}. 
\end{align*}

Let $u=L\partial_x\varphi+t\int |y|^{1-\alpha}F(\dq^y\varphi)\sdq^y\varphi_x\,dy$, which satisfies the linearized equation with error $\int |y|^{1-\alpha}F'(\dq^y\varphi)\dq^y\varphi\sdq^y\varphi_x\,dy$, which is clearly balanced. From Proposition \ref{Linearized equation energy estimates}, along with Gr\"onwall's lemma and the fact that $\eps\ll 1$, we have
\begin{align*}
 \|L\partial_x\varphi(t,x)\|_{L_x^2}\lesssim e^{C\int_0^tC(A(\tau))B(\tau)^2\,d\tau}\|u_0\|_{L_x^2}. 
\end{align*}
Along with the bootstrap assumptions, these readily imply that
\begin{equation}\label{Vector Field Energy Estimate}
    \|\varphi\|_{X}\lesssim \|\varphi(t)\|_{\dot{H}_x^{s_0}\cap\dot{H}_x^s} + \|L\partial_x\varphi(t)\|_{L_x^2}\lesssim \eps e^{C^2\eps^2\int_0^t\langle s\rangle^{-1}\,ds}\lesssim \eps \langle t\rangle^{C^2\eps^2}.
\end{equation}

\

\subsection{Vector field bounds} 
Proposition 2.1 from \cite{ITpax} implies that
\begin{align*}
    \|\varphi_\lambda\|^2_{L_x^{\infty}}&\lesssim \frac{1}{t\lambda^{\alpha-1}}(\|\varphi_\lambda\|_{L_x^2}\|L\partial_x\varphi_\lambda\|_{L_x^2}+\|\varphi_\lambda\|^2_{L_x^2}).
\end{align*}
When $\lambda\leq 1$,
\begin{align*}
  \||D_x|^{1-\delta}\varphi_\lambda\|_{L_x^{\infty}}
  &\lesssim \frac{1}{\sqrt{t}}\lambda^{\delta_1}(\|\lambda^{2-2\delta-2\delta_1}\varphi_\lambda\|^{1/2}_{L_x^2}\|L\partial_x\varphi_\lambda\|^{1/2}_{L_x^2}+\|\lambda^{1-\delta-\delta_1}\varphi_\lambda\|_{L_x^2}) \\
  &\lesssim \frac{1}{\sqrt{t}}\lambda^{\delta_1}\|\varphi\|_X
\end{align*}
and when $\lambda>1$,
\begin{align*}
  \||D_x|^{\frac{\alpha}{2}+\delta}\partial_x\varphi_\lambda\|_{L_x^{\infty}}
  &\lesssim \frac{1}{\sqrt{t}}\lambda^{-\delta}(\|\lambda^{\alpha+2+4\delta}\varphi_\lambda\|^{1/2}_{L_x^2}\|L\partial_x\varphi_\lambda\|^{1/2}_{L_x^2}+\|\lambda^{\frac{\alpha}{2}+1+2\delta}\varphi_\lambda\|_{L_x^2}) \lesssim \frac{1}{\sqrt{t}}\lambda^{-\delta}\|\varphi\|_X.
\end{align*}
By dyadic summation and Bernstein's inequality, we deduce the bound
\begin{equation}\label{Pointwise Vector Field Bound 2}
    \|\varphi\|_Y=\|\langle D_x\rangle^{2\delta+\frac{\alpha}{2}}|D_x|^{1-\delta}\varphi\|_{L_x^{\infty}}\lesssim\frac{\|\varphi\|_X}{\sqrt{t}}.
\end{equation}

By the localized dispersive estimate \cite[Proposition 5.1]{ITpax}, 
\begin{align*}
    |\varphi_\lambda(x)|^2\lesssim\frac{1}{|x-x_{\lambda}|t\lambda^{\alpha-2}}(\|L\varphi_\lambda\|_{L_x^2}+\lambda^{-1}\|\varphi_\lambda\|_{L_x^2})^2,
\end{align*}
 which implies that
 \begin{equation}\label{Pointwise Elliptic Estimate-IT}
    \|(1-\chi_{\lambda})\varphi_\lambda\|_{L_x^{\infty}}\lesssim \frac{\lambda^{\frac{1}{2}-\alpha}}{t}(\|L\partial_x\varphi_\lambda\|_{L_x^2}+\|\varphi_\lambda\|_{L_x^2})\lesssim\frac{\lambda^{\frac{1}{2}-\alpha}}{t}(\|\varphi\|_X+\|\varphi_\lambda\|_{L_x^2})
\end{equation}
 We also record the bound
 \begin{equation}\label{L2 Elliptic Estimate-IT}
    \|(1-\chi_{\lambda})\varphi_\lambda\|_{L_x^{\infty}}\lesssim \frac{\lambda^{\frac{1}{2}-\alpha}}{t}(\|L\partial_x\varphi_\lambda\|_{L_x^2}+\|\varphi_\lambda\|_{L_x^2})\lesssim\frac{\lambda^{\frac{1}{2}-\alpha}}{t}(\|\varphi\|_X+\|\varphi_\lambda\|_{L_x^2}),
\end{equation}
which follows directly from \cite[Proposition 5.1]{ITpax}.
\

To end this section we consider derivatives and difference quotients:

\begin{lemma}\label{Elliptic bounds for the derivative}
We have
  \begin{align*}
    \|\partial_x((1-\chi_{\lambda})\varphi_\lambda)\|_{L_x^{\infty}}+\|(1-\chi_{\lambda})\dq^y\varphi_\lambda\|_{L_x^{\infty}}&\lesssim  \frac{\lambda^{\frac{3}{2}-\alpha}}{t}\|\varphi\|_X,
  \end{align*}
  as well as the $L_x^2$-bounds
  \begin{align*}
    \|\partial_x((1-\chi_{\lambda})\varphi_\lambda)\|_{L_x^{2}}+\|(1-\chi_{\lambda})\dq^y\varphi_\lambda\|_{L_x^{2}}&\lesssim  \frac{\lambda^{1-\alpha}\|\varphi\|_X}{t}.
  \end{align*}
\end{lemma}
\begin{proof}
We use the bounds
\begin{align*}
    |\partial_x(\chi_{\lambda}(x/t)|&\lesssim t^{-1}\lambda^{1-\alpha}
\end{align*}
From \eqref{Pointwise Elliptic Estimate-IT} applied for $\partial_x\varphi$,
\begin{align*}
    \|\partial_x((1-\chi_{\lambda})\varphi_\lambda)\|_{L_x^{\infty}}&\lesssim \frac{1}{t}\lambda^{1-\alpha}\|\chi'_{\lambda}\varphi_\lambda\|_{L_x^{\infty}}+\|(1-\chi_{\lambda})\partial_x\varphi_\lambda\|_{L_x^{\infty}} \lesssim \frac{\lambda^{\frac32-\alpha}}{t}(\|\varphi\|_X+\|\varphi_\lambda\|_{L_x^2})
\end{align*}

For the bounds involving the difference quotient, from \ref{Pointwise Elliptic Estimate-IT} applied for $\dq^y\varphi$, we have
\begin{align*}
\|(1-\chi_{\lambda})\dq^y\varphi_\lambda\|_{L_x^\infty}&\lesssim \frac{\lambda^{\frac32-\alpha}}{t}(\|L\dq^y\varphi_\lambda\|_{L_x^2}+\lambda^{-1}\|\dq^y\varphi_\lambda\|_{L_x^2})\\
&\lesssim \frac{\lambda^{\frac32-\alpha}}{t}(\|\dq^y(L\varphi_\lambda)\|_{L_x^2}+\|\varphi_\lambda(x+y)\|_{L_x^2}+\|\varphi_\lambda\|_{L_x^2})\\
&\lesssim \frac{\lambda^{\frac32-\alpha}}{t}(\|L\partial_x\varphi_\lambda\|_{L_x^2}+\|\varphi_\lambda\|_{L_x^2})\\
&\lesssim \frac{\lambda^{\frac32-\alpha}}{t}(\|\varphi\|_X+\|\varphi_\lambda\|_{L_x^2})
\end{align*}
The other bounds are proved similarly.
\end{proof}

\subsection{Wave packets}

We construct wave packets as follows. Given the dispersion relation $a(\xi)$, the group velocity $v$ satisfies
\[
v = a'(\xi) = -c(\alpha)\alpha|\xi|^{\alpha-1},
\]
so we denote
\[
\xi_v = -\left(\frac{-v}{c(\alpha)\alpha}\right)^{\frac{1}{\alpha-1}}.
\]
Then we define the linear wave packet $\pax^v$ associated with velocity $v$ by
\[
\pax^v = a''(\xi_v)^{-\half} \chi(y) e^{it\phi(x/t)}, \qquad y = \frac{x - vt}{t^\half a''(\xi_v)^\half},
\]
where the phase $\phi$ is given by
\[
\phi(v) = v\xi_v - a(\xi_v),
\]
and $\chi$ is a unit bump function, such that $\int\chi(y)\,dy=1$.

We remark that we will typically apply frequency localization $\pax^v_\lambda = P_\lambda\mathbf{u}^v$ with $v \in J_\lambda$.

\

We observe that since
\[
\D_v (|a''(
\xi_v
)|^{-\half} ) \simeq (a''(\xi_v)^{-\half})^{\frac{4-3\alpha}{2-\alpha}}, 
\]
we may write
\begin{equation}\label{dpax}
\D_v\pax^v = - \tilde L \pax^v + (a''(\xi_v)^{-\half})^{\frac{2-2\alpha}{2-\alpha}}\pax^{v,II} = t^{\half} (a''(\xi_v)^{-\half})\pax^v + (a''(\xi_v)^{-\half})^{\frac{2-2\alpha}{2-\alpha}}\pax^{v,II} 
\end{equation}
where
 \[ \tilde{L}=t(\partial_x-i\phi'(x/t))\]
and $\pax^{v,II}$ has a similar wave packet form. We also recall from \cite[Lemmas 4.4, 5.10]{ITpax} the sense in which $\pax^v$ is a good approximate solution:

\begin{lemma}\label{l:paxsoln}
The wave packet $\pax^v$ solves an equation of the form
\[
(i\D_t - A(D)) \pax^v = t^{-\frac32}(L \pax^{v, I} + \rpax^v)
\]
where $\pax^{v, I}, \rpax^v$ have wave packet form,
\[
\pax^{v, I} \approx a''(\xi_v)^{-\half} \pax^v, \qquad \rpax^v \approx \xi_v^{-1}a''(\xi_v)^{-\half} \pax^v.
\]
\end{lemma}

\

The asymptotic profile at frequency $\lambda$ is meaningful when the associated spatial region $tJ_\lambda$ dominates the wave packet scale at frequency $\lambda$:
\[
\delta x \approx t^\half a''(\lambda)^\half \lesssim |tJ_\lambda| \approx t \lambda a''(\lambda).
\]
This corresponds to 
\[
t \gtrsim \lambda^{-2} a''(\lambda)^{-1} \approx \lambda^{-\alpha}.
\]
Accordingly we define
\[
\mathcal D = \{(t, v) \in \R^+ \times (-\infty,0) : v \in J_\lambda, \ t \gtrsim \lambda^{-\alpha} \}.
\]

\subsection{Wave packet testing}

In this section we establish estimates on the asymptotic profile function
\[
\gamma^\lambda(t, v) := \langle \varphi, \mathbf{u}^v_\lambda \rangle_{L^2_x}=\langle \varphi_\lambda, \mathbf{u}^v \rangle_{L^2_x}.
\]
We will see that $\gamma^\lambda$ essentially has support $v\in J_\lambda$.

\

We will also use the following crude bounds involving the higher regularity of $\gamma^\lambda$:
\begin{lemma}\label{Gamma bounds}
We have
    \begin{align*}
    \|\chi_\lambda \D_v^n \gamma^\lambda\|_{L^\infty}&\lesssim  t^{\half}(\lambda^{1-\alpha}+ t^\half\lambda^{1-\frac{\alpha}{2}})^n\|\varphi_{\lambda}\|_{L_x^{\infty}}, \\
      \|\chi_\lambda \D_v^n\gamma^\lambda\|_{L^2}&\lesssim  (t\lambda^{2-\alpha})^{\frac14}(\lambda^{1-\alpha}+t^\half\lambda^{1-\frac{\alpha}{2}})^n\|\varphi_{\lambda}\|_{L_x^{2}},
    \end{align*}
    and 
    \begin{align*}
\|\chi_\lambda\partial_v \gamma^\lambda\|_{L^\infty}&\lesssim t^{\frac14}\lambda^{-\frac12-\frac{\alpha}{4}}\|\varphi\|_X + t^\half \lambda^{1-\alpha}\|\varphi_\lambda\|_{L_x^{\infty}}.
    \end{align*}
    
\end{lemma}

\begin{proof}
Using the second form of $ \D_v \pax^v$ in \eqref{dpax}, we have
\[
|\chi_\lambda \D_v \gamma^\lambda| = |\chi_\lambda \langle \varphi_\lambda, \D_v \pax^v \rangle| \lesssim t^{\half}(t^\half\lambda^{1-\frac{\alpha}{2}} + \lambda^{1-\alpha})\|\varphi_{\lambda}\|_{L_x^{\infty}}
\]
where the $t^\half$ loss in front arises from the $L^1$ norm of the wave packet. Higher derivatives are obtained similarly, along with the $L^2$ estimates.

For the last bound, we use the first form of $ \D_v \pax^v$ in \eqref{dpax}. The contribution from the wave packet $\pax^{v,II}$ is easily estimated as above. For the remaining bound, Lemma 2.3 from \cite{ITpax} implies that
    \begin{align*}
        |\langle \varphi_\lambda,\tilde L \pax^v\rangle| &\lesssim (t\lambda^{2-\alpha})^{\frac14}\|\tilde{L}\varphi_\lambda\|_{L_x^2}\lesssim t^{\frac14}\lambda^{-\frac12-\frac{\alpha}{4}}\|\varphi_\lambda\|_X,
         \end{align*}
which finishes the proof.
\end{proof}

\subsubsection{Approximate profile}

We recall from \cite{ITpax} that $\gamma^\lambda$ provides a good approximation for the profile of $\varphi$. In our setting, we will also need to compare the profile with the differentiated flow $\D_x \varphi$. Define
\[
r^\lambda(t, x) = \varphi_\lambda(t, x) - t^{-\half}\gamma^\lambda(t, x/t)e^{-it\phi(x/t)}.
\]

\begin{lemma}\label{Error bounds 1}
Let $t \geq 1$. Then we have
\begin{equation*}
\begin{aligned}
\|\chi_\lambda(x/t) r^\lambda\|_{L^\infty_x} &\lesssim t^{-\frac34}\lambda^{\frac{\alpha}{4}-\frac{1}{2}} \| \tilde L \varphi_\lambda\|_{L^2_x}, \\
\|\chi_\lambda(x/t) \D_v r^\lambda\|_{L^\infty_x} &\lesssim t^{\frac14}\lambda^{\frac{\alpha}{4}-\frac{1}{2}} \| \tilde L \D_x\varphi_\lambda\|_{L^2_x} +  (\lambda^{1-\alpha} + t^\half \lambda^{1-\frac{\alpha}{2}})\|\varphi_\lambda \|_{L^\infty}.
\end{aligned}
\end{equation*}
\end{lemma}

\begin{proof}
The first estimate may be obtained from the proof of \cite[Proposition 4.7]{ITpax}. For the latter, we use the first representation in \eqref{dpax} to write
\begin{equation}\label{rexp}
e^{it\phi(v)} \D_v(\gamma(t, v) e^{-it(\phi(v)}) = t\langle \D_x\varphi_\lambda, \pax^v \rangle + \langle \varphi_\lambda, it(\phi'(\cdot/t) - \phi'(v)) \pax^v \rangle + (a''(\xi_v)^{-\half})^{\frac{2-2\alpha}{2-\alpha}}\langle \varphi_\lambda, \pax^{v, II} \rangle.
\end{equation}
To address the first term, we see that we may apply the undifferentiated estimate with $\D_x \varphi_\lambda$ in place of $\varphi_\lambda$. Precisely, we may apply the first estimate on
\[
\D_x\varphi_\lambda(t, x) - t^{-\half} \langle \D_x\varphi_\lambda, \pax^{x/t} \rangle e^{-it\phi(x/t)}.
\]
We estimate the third term of \eqref{rexp} via
\begin{equation*}
t^{-\half} (a''(\xi_v)^{-\half})^{\frac{2-2\alpha}{2-\alpha}}|\langle \varphi_\lambda, \pax^{v, II} \rangle| \lesssim \lambda^{1-\alpha}\|\varphi_\lambda \|_{L^\infty}.
\end{equation*}
It remains to estimate the middle term,
\[
 t^{-\half}|\chi_\lambda(v)\langle \varphi_\lambda, it(\phi'(\cdot/t) - \phi'(v)) \pax^v \rangle | \lesssim |\phi''(\lambda)| \cdot t^\half a''(\lambda)^\half \cdot \|\varphi_\lambda \|_{L^\infty} \lesssim t^\half \lambda^{1-\frac{\alpha}{2}} \|\varphi_\lambda \|_{L^\infty}
\]
\end{proof}

We denote
\[
\psi(t,x)=t^{-\half}\chi_\lambda(x/t)\gamma(t,x/t)e^{it\phi(x/t)}.
\]
We record some bounds for $\psi$:
\begin{lemma}\label{Psi bounds}
    Assume that $(t,v)\in\mathcal{D}$. We have
    \begin{align*}
        \|\D_x^n\psi(t,x)\|_{L_x^{\infty}}&\lesssim \|\varphi_\lambda\|_{L_x^{\infty}}\lambda^n\\
        \|\D_x^n \psi(t,x)\|_{L_x^{\infty}}&\lesssim \lambda^{\frac{1}{2}-\frac{\alpha}{4}}t^{-1/4}\|\varphi_\lambda\|_{L_x^{2}}\lambda^n\\
        \|\psi_x(t,x)\|_{L_x^{r}}&\lesssim \|\varphi_\lambda\|_{L_x^{r}}\lambda,\forall r\in[1,\infty]\\
    \end{align*}
\end{lemma}
\begin{proof}
     We have
      \begin{align*}
          \psi_x(t,x)=\frac{1}{\sqrt{t}}\frac{1}{t}\chi'_{\lambda}\gamma(t,x/t)e^{it\phi(x/t)}+\frac{1}{\sqrt{t}}\chi_{\lambda}\frac{1}{t}\gamma_v(t,x/t)e^{it\phi(x/t)}+\frac{1}{\sqrt{t}}\chi_{\lambda}\gamma(t,x/t)e^{it\phi(x/t)}i\phi'(x/t)
      \end{align*}
      By using Lemma \ref{Gamma bounds}, as well as the condition $ (t,v)\in\mathcal{D}$, we have
      \begin{align*}
        |\psi_x(t,x)|&\lesssim t^{-1/2}(t^{-1}t^{1/2}\|\varphi_\lambda\|_{L_x^{\infty}}(\lambda^{1-\alpha}+\lambda^{1-\frac{\alpha}{2}}t^{1/2})+t^{1/2}\|\varphi_\lambda\|_{L_x^{\infty}}\lambda)\\
        &\lesssim \|\varphi_\lambda\|_{L_x^{\infty}}(t^{-1}\lambda^{1-\alpha}+\lambda^{1-\alpha/2}t^{-1/2}+\lambda)\\
        &\lesssim \lambda\|\varphi_\lambda\|_{L_x^{\infty}}
      \end{align*}
      The other bounds can be deduced similarly.
\end{proof}

\

We also observe that on the wave packet scale, we may replace $\gamma(t, v)$ with $\gamma(t, x/t)$ up to acceptable errors. Indeed, let 
\[
\theta(t,x)=t^{-\half}\chi_\lambda(x/t)\gamma(t,v)e^{it\phi(x/t)}
,\]
and denote
\[
\beta^\lambda_v(t, x) = \theta(x,t)-\psi(t,x)=t^{-1/2}\chi_\lambda(x/t)(\gamma(t, v) - \gamma(t, x/t))e^{it\phi(x/t)},
\]

We are also going to need some bounds for $\theta$, that we record below:
\begin{lemma}\label{Theta bounds}
    Assume that $(t,v)\in\mathcal{D}$. Then, we have the bounds
    \begin{align*}
        \|\theta(t,x)\|_{L_x^{\infty}}&\lesssim \|\varphi_\lambda\|_{L_x^{\infty}}\\
        \|\theta_x(t,x)\|_{L_x^{r}}&\lesssim t^{1/r}\|\varphi_\lambda\|_{L_x^{\infty}}\lambda,\forall r\in[1,\infty]
    \end{align*}
\end{lemma}
\begin{proof}
      We have
      \begin{align*}
          \theta_x(t,x)=\frac{1}{\sqrt{t}}\frac{1}{t}\chi_{\lambda}'(x/t)\gamma(t,v)e^{it\phi(x/t)}+\frac{1}{\sqrt{t}}\chi_{\lambda}(x/t)\gamma(t,v)e^{it\phi(x/t)}i\phi'(x/t)
      \end{align*}
      By using Lemma \ref{Gamma bounds}, we have
      \begin{align*}
        |\theta_x(t,x)|&\lesssim t^{-1/2}(t^{-1}\lambda^{1-\alpha}t^{1/2}\|\varphi_\lambda\|_{L_x^{\infty}}+t^{1/2}\|\varphi_\lambda\|_{L_x^{\infty}}\lambda) \lesssim \|\varphi_\lambda\|_{L_x^{\infty}}\lambda
      \end{align*}
      The other bounds can be deduced similarly.
\end{proof}
\begin{lemma}\label{Beta bounds} 
Let $v\in J_\lambda$, and $(t,v)\in\mathcal{D}$. Then, for every $y\neq 0$ and $x$ such that $\displaystyle|x-vt|\lesssim\delta x=t^{1/2}\lambda^{\frac{\alpha}{2}-1}$, we have the bound 
\begin{align*}
|\dq^y\beta_v|&\lesssim t^{-3/4}\lambda^{\frac{\alpha}{4}-\half}\|\varphi\|_X
\end{align*}
\end{lemma}
\begin{proof}
We have
\[
\delta^y \beta_v = -t^{-1/2}\delta^y(\gamma(t, \cdot/t)) \chi_\lambda((x+y)/t)e^{it\phi((x + y)/t)} + t^{-1/2}(\gamma(t, v) - \gamma(t, x/t)) \delta^y (\chi_\lambda(\cdot/t)e^{it\phi(\cdot/t)}).
\]
The Mean Value Theorem ensures that
\[
|\delta^y(\gamma(t, \cdot/t))| \lesssim t^{-1} \|\D_v \gamma\|_{L^\infty},
\]
and that
\begin{align*}
|\dq^y\beta_v|&\lesssim t^{-1/2}(t^{-1}\|\partial_v\gamma\|_{L^\infty}+t^{-1/2}\lambda^{\frac{\alpha}{2}-1}\|\partial_v\gamma\|_{L^\infty}(t^{-1}\lambda^{1-\alpha}+\lambda))\\
&\lesssim t^{-1}\|\partial_v\gamma\|_{L^\infty}(t^{-1/2}+\lambda^{\alpha/2})\lesssim t^{-1}\lambda^{\alpha/2}(t^{\frac14}\lambda^{-\frac12-\frac{\alpha}{4}}\|\varphi\|_X + t^\half \|\varphi_\lambda\|_{L_x^{\infty}})\\
&\lesssim t^{-3/4}\lambda^{\frac{\alpha}{4}-\half}\|\varphi\|_X+t^{-1/2}\lambda^{\alpha/2}\|\varphi_\lambda\|_{L_x^\infty}\lesssim t^{-3/4}\lambda^{\frac{\alpha}{4}-\half}\|\varphi\|_X
\end{align*}
\end{proof}

\subsection{Bounds for $Q$}

Write, slightly abusing notation, 
\[
Q(\varphi) = Q(\varphi, \overline{\varphi}, \varphi) := \frac13 \int\frac{1}{|y|^{\alpha-1}} \sgn(y) \cdot |\dq^y\varphi|^2\dq^y\varphi\, dy.
\]

\begin{lemma}\label{Cubic difference bound}
For $\alpha>1$, we have the difference estimates 
\begin{equation*}\begin{aligned}
\|Q(\varphi_1) - Q(\varphi_2)\|_{L_x^\infty} &\lesssim \|\D_x(\varphi_1,\varphi_2)\|
_{L_x^{\infty}}\|(\varphi_1,\varphi_2)\|_{W_x^{1,\infty}} \|\partial_x(\varphi_1-\varphi_2)\|_{L_x^{\infty}},
\end{aligned}\end{equation*}
\begin{equation*}\begin{aligned}
\|Q(\varphi_1) - Q(\varphi_2)\|_{L_x^2}&\lesssim \|\D_x (\varphi_1,\varphi_2)\|_{L_x^2}\|(\varphi_1,\varphi_2)\|_{W_x^{1,\infty}}\|\partial_x(\varphi_1-\varphi_2)\|_{L_x^\infty},
\end{aligned}\end{equation*}
while for $\alpha<1$ we have
\begin{equation*}\begin{aligned}
\|Q(\varphi_1) - Q(\varphi_2)\|_{L_x^\infty} &\lesssim \|(\varphi_1,\varphi_2)\|
_{L_x^{\infty}}\|(\varphi_1,\varphi_2)\|_{W_x^{1,\infty}} \|\partial_x(\varphi_1-\varphi_2)\|_{L_x^{\infty}},
\end{aligned}\end{equation*}
\begin{equation*}\begin{aligned}
\|Q(\varphi_1) - Q(\varphi_2)\|_{L_x^2}&\lesssim \| (\varphi_1,\varphi_2)\|_{L_x^2}\|(\varphi_1,\varphi_2)\|_{W_x^{1,\infty}}\|\partial_x(\varphi_1-\varphi_2)\|_{L_x^\infty},
\end{aligned}\end{equation*}
Moreover, for $\alpha>1$ we also have the estimates 
\begin{equation*}\begin{aligned}
\|Q(\varphi_1) - Q(\varphi_2)\|_{L_x^\infty} &\lesssim \||D_x|^{1-\delta}(\varphi_1,\varphi_2)\|
_{W_x^{2\delta,\infty}}\||D_x|^{\alpha-1}(\varphi_1,\varphi_2)\|_{L_x^\infty} \|\partial_x(\varphi_1-\varphi_2)\|_{L_x^{\infty}},
\end{aligned}\end{equation*}
\begin{equation*}\begin{aligned}
\|Q(\varphi_1) - Q(\varphi_2)\|_{L_x^2} &\lesssim \||D_x|^{1-\delta}(\varphi_1,\varphi_2)\|
_{H_x^{2\delta}}\||D_x|^{\alpha-1}(\varphi_1,\varphi_2)\|_{L_x^\infty} \|\partial_x(\varphi_1-\varphi_2)\|_{L_x^{\infty}},
\end{aligned}\end{equation*}
while for $\alpha<1$ we have
\begin{equation*}\begin{aligned}
\|Q(\varphi_1) - Q(\varphi_2)\|_{L_x^\infty} &\lesssim \||D_x|^{1-\delta}(\varphi_1,\varphi_2)\|
_{W_x^{\delta,\infty}}\||D_x|^{\alpha-\delta}(\varphi_1,\varphi_2)\|_{W_x^{\delta,\infty}} \|\partial_x(\varphi_1-\varphi_2)\|_{L_x^{\infty}},
\end{aligned}\end{equation*}
\begin{equation*}\begin{aligned}
\|Q(\varphi_1) - Q(\varphi_2)\|_{L_x^2} &\lesssim \||D_x|^{1-\delta}(\varphi_1,\varphi_2)\|
_{H_x^{\delta}}\||D_x|^{\alpha-\delta}(\varphi_1,\varphi_2)\|_{H_x^\delta} \|\partial_x(\varphi_1-\varphi_2)\|_{L_x^{\infty}},
\end{aligned}\end{equation*}
\end{lemma}

\begin{proof}
We only prove the first two estimates in the case $\alpha>1$, as the other ones are similar. 

Write
\[
Q(\varphi_1) - Q(\varphi_2) = \int_{|y|\leq 1} + \int_{|y|> 1}
\]
where the integrand may be written
\begin{align*}
\sgn(y)(|\dq^y\varphi_1|^2\dq^y\overline{\varphi_1}- &|\dq^y\varphi_2|^2\dq^y\overline{\varphi_2}) \\
&= \sgn(y)(\dq^y(\varphi_1-\varphi_2)(|\dq^y\varphi_1|^2+|\dq^y\varphi_2|^2)+\dq^y(\overline{\varphi_1}-\overline{\varphi_2})\dq^y\varphi_1\dq^y\varphi_2).
\end{align*}
The first integral contributes to the two estimates respectively,
\begin{equation*}\begin{aligned}
\bigg| \int_{|y|\leq 1} \bigg| \lesssim \|\D_x(\varphi_1,\varphi_2)\|_{L_x^\infty}^2\|\partial_x(\varphi_1-\varphi_2)\|_{L_x^{\infty}}
\end{aligned}\end{equation*}
and
\[
\bigg\| \int_{|y|\leq 1}\bigg\|_{L^2} \lesssim \|\D_x(\varphi_1,\varphi_2)\|_{L_x^2}\|\D_x(\varphi_1,\varphi_2)\|_{L_x^\infty} \|\partial_x(\varphi_1-\varphi_2)\|_{L_x^{\infty}}.
\]

For the second, 
\begin{equation*}\begin{aligned}
\bigg\| \int_{|y|> 1}\bigg\|_{L_x^{\infty}} &\lesssim \||\partial_x(\varphi_1,\varphi_2)\|_{L_x^{\infty}}\|(\varphi_1,\varphi_2)\|_{L_x^\infty}\|\partial_x(\varphi_1-\varphi_2)\|_{L_x^{\infty}}
\end{aligned}\end{equation*}
and 
\begin{equation*}\begin{aligned}
\bigg\| \int_{|y|> 1}\bigg\|_{L^{2}} &\lesssim\||\partial_x(\varphi_1,\varphi_2)\|_{L_x^{2}}\|(\varphi_1,\varphi_2)\|_{L_x^\infty}  \|\partial_x(\varphi_1-\varphi_2)\|_{L_x^{\infty}}.
\end{aligned}\end{equation*}
\end{proof}

We will be considering separately the balanced and unbalanced components of $Q$. Precisely, we denote the diagonal set of frequencies by $\mathcal D$ and write
\begin{equation*}
\begin{aligned}
Q(\varphi, \varphi, \varphi) &= \sum_{(\lambda_1, \lambda_2, \lambda_3, \lambda) \in \mathcal D} Q(\varphi_{\lambda_1}, \varphi_{\lambda_2}, \varphi_{\lambda_3}) + \sum_{(\lambda_1, \lambda_2, \lambda_3, \lambda) \notin \mathcal D} Q(\varphi_{\lambda_1}, \varphi_{\lambda_2}, \varphi_{\lambda_3}) \\
&= Q^{bal}(\varphi, \varphi, \varphi) + Q^{unbal}(\varphi, \varphi, \varphi) = Q^{bal}(\varphi) + Q^{unbal}(\varphi).
\end{aligned}
\end{equation*}

\

The unbalanced portion of $Q$ satisfies the better bound as follows: 

\begin{lemma}\label{Unbalanced cubic estimate}
$Q^{unbal}$ satisfies the bounds
\[
\|\chi_\lambda^1\partial_xP_\lambda Q^{unbal}(\varphi)\|_{L_x^\infty} \lesssim \lambda^{\max\{\frac{1-\alpha}{2},0\}-\delta}\frac{\|\varphi\|^3_X}{t^2}
\]
and
\[
\|\chi_\lambda^1\partial_xP_\lambda Q^{unbal}(\varphi)\|_{L_x^2} \lesssim \lambda^{-\delta}\frac{\|\varphi\|^3_X}{t^{3/2}},
\]
where $\chi^1_{\lambda}$ is a cut-off widening $\chi_{\lambda}$.
\end{lemma}

\begin{proof}
We shall denote
\begin{align*}
    I_{\lambda_1,\lambda_2,\lambda_3}&=\int_{\mathbb{R}}\frac{1}{|y|^{\alpha-1}}\sgn(y)\dq^y\varphi_{\lambda_1}\dq^y\varphi_{\lambda_2}\dq^y\varphi_{\lambda_3}\,dy
\end{align*}
and consider two cases in the frequency sum for $\partial_x P_\lambda Q^{unbal}$. 

First we consider the case in which we have two low separated frequencies. We assume without loss of generality that $\lambda_3=\lambda$ and $\lambda_1<\lambda_2\ll \lambda$.
We analyze the case $\alpha>1$.
Here, the elliptic estimates will be applied for the factor $\varphi_{\lambda_2}$. Precisely, from Lemma \ref{Trilinear integral estimate-v0} and estimates \ref{Pointwise Vector Field Bound 2}, \ref{Pointwise Elliptic Estimate-IT}, and \ref{L2 Elliptic Estimate-IT}, we get that 
\begin{align*}
    \left\|\chi^1_{\lambda}I_{\lambda_1,\lambda_2,\lambda_3}\right\|_{L_x^{\infty}}&\lesssim \frac{\lambda_2^{\frac32-\alpha}}{t}\|\varphi\|_X(\lambda_1^{1-2\delta}+\lambda_1)\|\varphi_{\lambda_1}\|_{L_x^\infty}\lambda_3^{\alpha+\delta-1}\|\varphi_{\lambda_3}\|_{L_x^\infty}\\
    &\lesssim \frac{\lambda_2^{\frac32-\alpha}}{t}\|\varphi\|_X\lambda_1^{\frac{\alpha-1}{2}+\delta}(\lambda_1^{\frac{3-\alpha}{2}-3\delta}+\lambda_1^{\frac{3-\alpha}{2}-\delta})\|\varphi_{\lambda_1}\|_{L_x^\infty}\lambda^{-\left(2-\frac{\alpha}{2}+2\delta\right)}\lambda^{1+\frac{\alpha}{2}+3\delta}\|\varphi_{\lambda}\|_{L_x^\infty}\\
    &\lesssim \lambda_2^{\frac32-\alpha}\lambda_1^{\frac{\alpha-1}{2}+\delta}\lambda^{-\left(2-\frac{\alpha}{2}+2\delta\right)}\frac{\|\varphi\|^3_X}{t^2}.
\end{align*}
When $\alpha<1$, we also apply the elliptic estimates for the factor $\varphi_{\lambda_1}$. Precisely, from Lemma \ref{Trilinear integral estimate-v0} and estimates \ref{Pointwise Vector Field Bound 2}, \ref{Pointwise Elliptic Estimate-IT}, and \ref{L2 Elliptic Estimate-IT}, we get that
\begin{align*}
    \left\|\chi^1_{\lambda}I_{\lambda_1,\lambda_2,\lambda_3}\right\|_{L_x^{\infty}}&\lesssim \frac{\lambda_1^{\frac32-\alpha}}{t}\|\varphi\|_X(\lambda_2^{2\delta}+\lambda_2^{4\delta})\|\varphi_{\lambda_2}\|_{L_x^\infty}\lambda_3^{\alpha-3\delta}\|\varphi_{\lambda_3}\|_{L_x^\infty}\\
    &\lesssim \frac{\lambda_1^{\frac32-\alpha}}{t}\|\varphi\|_X\lambda_2^{-\left(1+\delta\right)}(\lambda_2^{1+\frac{\alpha}{2}+3\delta}+\lambda_2^{1+\frac{\alpha}{2}+5\delta})\|\varphi_{\lambda_2}\|_{L_x^\infty}\lambda^{-\left(1-\frac{\alpha}{2}+4\delta\right)}\lambda^{1+\frac{\alpha}{2}+\delta}\|\varphi_{\lambda}\|_{L_x^\infty}\\
    &\lesssim \lambda_1^{\frac32-\alpha}\lambda_2^{-\left(1+\delta\right)}\lambda^{-\left(1-\frac{\alpha}{2}+2\delta\right)}\frac{\|\varphi\|^3_X}{t^2}.
\end{align*}

By using dyadic summation in $\lambda_1$ and $\lambda_2$, we deduce that 

\begin{align*}
    \left\|\chi^1_{\lambda}\partial_x\sum_{\substack{
    \lambda_1<\lambda_2\ll \lambda}}I_{\lambda_1,\lambda_2,\lambda_3}\right\|_{L_x^{\infty}}&\lesssim \lambda^{\max\{\frac{1-\alpha}{2},0\}-\delta}\frac{\|\varphi\|_X^3}{t^{2}}.
\end{align*}
Similarly, we deduce that 
\begin{align*}
    \left\|\chi^1_{\lambda}\partial_x\sum_{\substack{
    \lambda_1<\lambda_2\ll \lambda}}I_{\lambda_1,\lambda_2,\lambda_3}\right\|_{L_x^{2}}&\lesssim \lambda^{-\delta}\frac{\|\varphi\|_X^3}{t^{3/2}}
\end{align*}

We now analyze the situation in which $\lambda_1,\lambda_2\gtrsim \lambda$, and $\lambda_1$ and $\lambda_2$ are comparable and both separated from $\lambda$. Thus, we will be able to use $\lambda_1$ and $\lambda_2$ interchangeably.  We replace $\chi^1_{\lambda}$ by $\tilde{\chi}_{\lambda}$, which has double support, and equals $1$ on a comparably-sized neighbourhood of the support of $\chi^1_{\lambda}$. We write
\begin{align*}
    \chi^1_{\lambda}\partial_xP_\lambda=\chi^1_{\lambda}\partial_xP_\lambda\tilde{\chi}_{\lambda}+\chi^1_{\lambda}\partial_xP_\lambda(1-\tilde{\chi}_{\lambda}).
\end{align*}

For the first term, when $\alpha>1$ using Lemma \ref{Trilinear integral estimate-v0}, along with estimates \ref{Pointwise Vector Field Bound 2}, \ref{Pointwise Elliptic Estimate-IT}, \ref{L2 Elliptic Estimate-IT}, we get the bounds 
\begin{align*}
    \left\|\chi^1_{\lambda}P_\lambda\tilde{\chi}_{\lambda}I_{\lambda_1,\lambda_2,\lambda_3}\right\|_{L_x^{\infty}}&\lesssim \lambda_2^{1/2+\delta}\frac{\lambda_3^{1-2\delta}+\lambda_3}{t}\|\varphi\|_X\|\varphi_{\lambda_2}\|_{L_x^\infty}\|\varphi_{\lambda_3}\|_{L_x^\infty}\\
    &\lesssim \lambda_2^{-1-3\delta/2}\lambda_3^{\delta/2}\frac{\|\varphi\|_X}{t}(\lambda_3^{1-5\delta/2}+\lambda_3^{1-\delta/2})\|\varphi_{\lambda_3}\|_{L_x^\infty}\lambda_2^{\frac{3}{2}+5\delta/2}\|\varphi_{\lambda_2}\|_{L_x^\infty}\\
    &\lesssim\lambda_2^{-1-\frac{3\delta}{2}}\lambda_3^{\delta/2}\frac{\|\varphi\|^3_X}{t^2},
\end{align*}
and

\begin{align*}
    \left\|\chi^1_{\lambda}P_\lambda\tilde{\chi}_{\lambda}I_{\lambda_1,\lambda_2,\lambda_3}\right\|_{L_x^{\infty}}&\lesssim \lambda_2^{1/2+\delta}\frac{\lambda_3^{1-2\delta}+\lambda_3}{t}\|\varphi\|_X\|\varphi_{\lambda_2}\|_{L_x^\infty}\|\varphi_{\lambda_3}\|_{L_x^\infty}\\
    &\lesssim \lambda_2^{-\frac{\alpha+1}{2}-3\delta/2}\lambda_3^{\delta/2}\frac{\|\varphi\|_X}{t}(\lambda_3^{1-5\delta/2}+\lambda_3^{1-\delta/2})\|\varphi_{\lambda_3}\|_{L_x^\infty}\lambda_2^{\frac{3}{2}+5\delta/2}\|\varphi_{\lambda_2}\|_{L_x^\infty}\\
    &\lesssim\lambda_2^{-\frac{\alpha+1}{2}-\frac{3\delta}{2}}\lambda_3^{\delta/2}\frac{\|\varphi\|^3_X}{t^2},
\end{align*}
when $\alpha<1$.

We similarly get the following $L_x^2$ bound:
\begin{align*}
    \left\|\chi^1_{\lambda}P_\lambda\tilde{\chi}_{\lambda}I_{\lambda_1,\lambda_2,\lambda_3}\right\|_{L_x^{2}}&\lesssim \lambda_2^{\delta}\frac{\lambda_3^{1-2\delta}+\lambda_3}{t}\|\varphi\|_X\|\varphi_{\lambda_2}\|_{L_x^\infty}\|\varphi_{\lambda_3}\|_{L_x^\infty}\\
    &\lesssim \lambda_2^{-1-3\delta/2}\lambda_3^{\delta/2}\frac{\|\varphi\|_X}{t}(\lambda_3^{1-5\delta/2}+\lambda_3^{1-\delta/2})\|\varphi_{\lambda_3}\|_{L_x^\infty}\lambda_2^{1+5\delta/2}\|\varphi_{\lambda_2}\|_{L_x^\infty}\\
    &\lesssim\lambda_2^{-1-3\delta/2}\lambda_3^{\delta/2}\frac{\|\varphi\|^3_X}{t^2}.
\end{align*}
By using dyadic summation in $\lambda_1$, $\lambda_2$, and $\lambda_3$ (and by using the fact that $\lambda_1$ and $\lambda_2$ are close), we deduce the bounds
\begin{align*}
   \left\|\chi^1_{\lambda}\partial_xP_\lambda\tilde{\chi}_{\lambda}\sum_{\substack{
    \lambda_3\lesssim \lambda_2,\lambda_1\simeq\lambda_2\gtrsim \lambda}}I_{\lambda_1,\lambda_2,\lambda_3}\right\|_{L_x^\infty}&\lesssim  \lambda^{\max\{\frac{1-\alpha}{2},0\}-\delta}\frac{1}{t^2}\|\varphi\|^3_X\\
    \left\|\chi^1_{\lambda}\partial_xP_\lambda\tilde{\chi}_{\lambda}\sum_{\substack{
    \lambda_3\lesssim \lambda_2,\lambda_1\simeq\lambda_2\gtrsim \lambda}}I_{\lambda_1,\lambda_2,\lambda_3}\right\|_{L_x^{2}}&\lesssim  \lambda^{-\delta}\frac{1}{t^2}\|\varphi\|^3_X. 
\end{align*}

We look at the second term. For every $N$, we know that
\begin{align*}
    \|\chi^1_{\lambda}\partial_xP_\lambda(1-\tilde{\chi}_{\lambda})\|_{L^2\rightarrow L^2}, \|\chi^1_{\lambda}\partial_xP_\lambda(1-\tilde{\chi}_{\lambda})\|_{L^\infty\rightarrow L^\infty}&\lesssim \frac{\lambda^{1-N}}{t^{N}}
\end{align*}
 By carrying out a similar analysis as above, along with Lemma \ref{Trilinear integral estimate-v0} and dyadic summation, we deduce that the contributions corresponding to these terms are also acceptable. 
\end{proof}
\begin{lemma}\label{Semiclassical computation}
We have
\begin{align*}\chi_\lambda((x/t))^3Q(e^{it\phi(x/t)})&=(\chi_\lambda(x/t))^3e^{it\phi(x/t)}q(\phi'(x/t))+h(\lambda,t),
\end{align*}
where for every $a\in(0,1)$
\begin{align*}
|h(\lambda,t)|&\lesssim \frac{\lambda^{6-3\alpha}+\lambda^{5-2\alpha}}{t^{2-(4-\alpha)a}}+\frac{\lambda^{3-\alpha}}{t^{1-(2-\alpha)a}}+\frac{1}{t^{(1+\alpha)a}}
\end{align*},
when $\alpha>1$, and
\begin{align*}
|h(\lambda,t)|&\lesssim \frac{\lambda^{4-2\alpha}}{t^{2-(3-\alpha)a}}+\frac{\lambda^{3-\alpha}}{t^{1-(2-\alpha)a}}+\frac{1}{t^{(1+\alpha)a}}
\end{align*}
for $\alpha<1$.
\end{lemma}
\begin{proof}
We write
    \begin{align*}
       e^{-it\phi(x/t)} \dq^ye^{it\phi(x/t)}&=\frac{e^{iy\phi'(x/t)}(e^{it/2\phi''(c_{x,y}/t)y^2/t^2}-1)}{y} + \frac{e^{iy\phi'(x/t)}-1}{y} =: a + b,
    \end{align*}
    where $c_{x,y}$ is between $x$ and $x+y$.
We now use the fact that $x/t$ belongs to the support of $\chi_\lambda$. We have
\begin{align*}
        |\chi_\lambda(x/t)||b| &\lesssim \lambda.
    \end{align*}
    Moreover, when $|y|\leq t^a$, $\displaystyle|c_{x,y}/t-x/t|\leq|y/t|\leq t^{a-1}$. This implies that $c_{x,y}/t$ belongs to the support of the enlarged cut-off $\chi_\lambda^1$, hence $\phi''(c_{x,y}/t)\simeq\lambda^{2-\alpha}$. We note the bound
    \begin{align*}
        |\chi_\lambda(x/t)||a|&\lesssim |\chi_\lambda(x/t)||y/(2t)\phi''(c_{x,y}/t)|\left|\frac{e^{\pm i\phi''(c_{x,y}/t)y^2/(2t)}-1}{\phi''(c_{x,y}/t)y^2/(2t)}\right| \lesssim \lambda^{2-\alpha} t^{a-1}
    \end{align*}
    Thus, we have the bounds
    \begin{equation}\label{ab-Bounds1}
    \begin{aligned}
|\chi_\lambda(x/t)||a|&\lesssim \lambda^{2-\alpha} t^{a-1}\\
|\chi_\lambda(x/t)||b|&\lesssim \lambda
\end{aligned}
    \end{equation}
    We also note the cruder bounds
    \begin{equation}\label{ab-Bounds2}
    \begin{aligned}
|\chi_\lambda(x/t)||a|+
|\chi_\lambda(x/t)||b|&\lesssim \frac{1}{|y|}
\end{aligned}
    \end{equation}
    We write
    \begin{align*}
(\chi_\lambda(x/t))^3Q(e^{it\phi(x/t)})&=(\chi_\lambda(x/t))^3e^{it\phi(x/t)}\int\frac{1}{|y|^{\alpha-1}} \left|b\right|^2b\,dy\\
&\quad +(\chi_\lambda(x/t))^3e^{it\phi(x/t)}\int\frac{1}{|y|^{\alpha-1}}( a^2\overline{a}+a^2\overline{b}+2|a|^2b+2a|b|^2+b^2\overline{a})\,dy\\
&:=T_1+T_2
    \end{align*}
We note that
\begin{align*}
    T_1&=(\chi_\lambda(x/t))^3e^{it\phi(x/t)}\int\frac{1}{|y|^{\alpha-1}} \left|b\right|^2b\,dy=(\chi_\lambda(x/t))^3e^{it\phi(x/t)}q(\phi'(x/t)),
\end{align*}
so we only need to analyze $T_2$.

We first bound the contribution over the region $|y|\leq t^a$, which we shall denote by $T_2^1$. We denote the contribution over the region $|y|>t^a$ by $T_2^2$. We have 
    \begin{align*}
T_2^1&=(\chi_\lambda(x/t))^3e^{it\phi(x/t)}\int_{|y|\leq t^a}\frac{1}{|y|^{\alpha-1}}( a^2\overline{a}+a^2\overline{b}+2|a|^2b)\,dy\\
&+(\chi_\lambda(x/t))^3e^{it\phi(x/t)}\int_{|y|\leq t^a}\frac{1}{|y|^{\alpha-1}}( 2a|b|^2+b^2\overline{a})\,dy:=T_{2a}+T_{2b},
    \end{align*}
    \ref{ab-Bounds1} implies that 
    \begin{align*}
|T_{2a}|&\lesssim |\chi_{\lambda}(x/t)|^3\int_{|y|\leq t^a}\frac{1}{|y|^{\alpha-1}}(|a|^3+|a|^2|b|)\,dy\lesssim \int_{|y|\leq t^a}\frac{1}{|y|^{\alpha-1}}\frac{\lambda^{4-2\alpha}}{t^{2-2a}}\left(\frac{\lambda^{2-\alpha}}{t^{1-a}}+\lambda\right)\\
&\lesssim \int_{|y|\leq t^a}\frac{1}{|y|^{\alpha-1}}\frac{\lambda^{6-3\alpha}+\lambda^{5-2\alpha}}{t^{2-2a}}\,dy\lesssim t^{a(2-\alpha)}\frac{\lambda^{6-3\alpha}+\lambda^{5-2\alpha}}{t^{2-2a}}\lesssim \frac{\lambda^{6-3\alpha}+\lambda^{5-2\alpha}}{t^{2-(4-\alpha)a}},
    \end{align*}
    when $\alpha>1$, and
       \begin{align*}
|T_{2a}|&\lesssim |\chi_{\lambda}(x/t)|^3\int_{|y|\leq t^a}\frac{1}{|y|^{\alpha-1}}(|a|^3+|a|^2|b|)\,dy\lesssim \int_{|y|\leq t^a}\frac{1}{|y|^{\alpha}}\frac{\lambda^{2-\alpha}}{t^{1-a}}\left(\frac{\lambda^{2-\alpha}}{t^{1-a}}+\lambda\right)\\
&\lesssim \int_{|y|\leq t^a}\frac{1}{|y|^{\alpha}}\left(\frac{\lambda^{4-2\alpha}}{t^{2-2a}}+\frac{\lambda^{3-\alpha}}{t^{1-a}}\right)\,dy\lesssim t^{a(1-\alpha)}\left(\frac{\lambda^{4-2\alpha}}{t^{2-2a}}+\frac{\lambda^{3-\alpha}}{t^{1-a}}\right)\lesssim \frac{\lambda^{4-2\alpha}}{t^{2-(3-\alpha)a}}+\frac{\lambda^{3-\alpha}}{t^{1-(2-\alpha)a}},
    \end{align*}
    when $\alpha<1$.
    
  \ref{ab-Bounds1} and \ref{ab-Bounds2} imply the bound
  \begin{align*}
|\chi_{\lambda}(x/t)|^3|b|^2|a|&=|\chi_{\lambda}(x/t)|^3\left|b\right|^2|y/(2t)\phi''(c_{x,y}/t)|\left|\frac{e^{\pm i\phi''(c_{x,y}/t)y^2/(2t)}-1}{\phi''(c_{x,y}/t)y^2/(2t)}\right|\\
&\lesssim\lambda\frac{1}{|y|}|\chi_{\lambda}(x/t)||y/(2t)\phi''(c_{x,y}/t)|\lesssim \frac{\lambda^{3-\alpha}}{t}
  \end{align*}
    It follows that $T_{2b}$ satisfies the bound 
    \begin{align*}
|T_{2b}|&\lesssim |\chi_\lambda(x/t)|^3\int_{|y|\leq t^a}\frac{1}{|y|^{\alpha-1}}\left|b\right|^2|a|\,dy\lesssim\int_{|y|\leq t^a}\frac{1}{|y|^{\alpha-1}}\frac{\lambda^{3-\alpha}}{t}\,dy\lesssim\frac{\lambda^{3-\alpha}}{t^{1-(2-\alpha)a}}
    \end{align*}

For $T_2^2$, \ref{ab-Bounds2} implies that
  \begin{align*}
|T_2^2|\lesssim \int_{|y|>t^a}\frac{1}{|y|^{\alpha+2}}\,dy\lesssim \frac{1}{t^{(1+\alpha)a}}.
  \end{align*}
\end{proof}
\

\subsection{The asymptotic equation for $\gamma$}

Here we prove the following:

\begin{proposition}\label{Asymptotic equation}
Let $v\in J_\lambda$. Under the assumption $(t,v)\in\mathcal{D}$, we have
\begin{align*}
    \dot{\gamma}(t,v)=iq(\xi_v)\xi_vt^{-1}\gamma(t,v)|\gamma(t,v)|^2+f(t,v),
\end{align*}
where
\begin{align*}
    |f(t,v)|&\lesssim \lambda^{-\delta}g(\lambda,\alpha)t^{-1-\delta+C\eps^2}\eps,
\end{align*} 
where $g=g(\lambda,\alpha)$ is a sum of powers of $\lambda$ that might also depend on $\alpha$.

Moreover, we also have the following $L_v^2$ bound:
\begin{align*}
    \|f(t,v)\|_{L_v^2(J_{\lambda})}&\lesssim (\lambda^{-\delta}+\lambda^{-\frac32-\delta})t^{-1-\delta+C\eps^2}\eps.
\end{align*}
\end{proposition}
\begin{proof} 
We have
\begin{align*}
  \dot{\gamma}(t,v)=\left\langle \dot{\varphi},\mathbf{u}^v_\lambda\right\rangle+\left\langle \varphi,\dot{\mathbf{u}^v_\lambda}\right\rangle  &=\left\langle P_\lambda A_\varphi \varphi,\pax^v\right\rangle + i\langle \varphi_\lambda,(i\D_t - A(D))\mathbf{u}^v\rangle:=I_1+I_2.
\end{align*}

We first analyze $I_2$. We use Lemma~\ref{l:paxsoln} to write
\[
(i\D_t - A(D)) \pax^v = t^{-\frac32}(L \pax^{v, I} + \rpax^v)
\]
\begin{align*}
   \left|\langle \varphi_\lambda,(i\D_t - A(D))\mathbf{u}^v\rangle\right|&\lesssim t^{-\frac32}(\| L\varphi_\lambda\|_{L_x^2} \cdot \lambda^{1-\frac{\alpha}{2}}\lambda^{\half-\frac{\alpha}{4}}t^{1/4} + \|\varphi_\lambda\|_{L_x^2}  \cdot \lambda^{-\alpha/2}\lambda^{\half-\frac{\alpha}{4}}t^{1/4})\\
   &\lesssim \lambda^{\half-\frac{3\alpha}{4}} t^{-5/4}\|\varphi\|_X\\
   &\lesssim \lambda^{\frac{1-\alpha}{2}-\delta} t^{-1-\delta}\eps t^{C\eps^2}
\end{align*}
and 
\begin{align*}
   \|\chi_\lambda \langle \varphi_\lambda,(i\D_t - A(D))\mathbf{u}^v\rangle\|_{L_v^2}&\lesssim t^{-3/2}\left(\|L\varphi_\lambda\|_{L_x^2}\lambda^{1-\frac{\alpha}{2}}+\|\varphi_\lambda\|_{L_x^2}\lambda^{-\frac{\alpha}{2}}\right)\\
   &\lesssim \lambda^{-\alpha/4} t^{-5/4}t^{-1/4}\lambda^{-\alpha/4}\|\varphi\|_X\\
   &\lesssim \lambda^{-\delta} t^{-1-\delta}\eps t^{C\eps^2}
\end{align*}
(we have used the condition $\displaystyle (t,v)\in\mathcal{D}$.)

\

In the remaining part of this section we shall analyze the term $I_1$. We first exchange $F$ for its principal quadratic term, expanding
\begin{align*}
   F(\dq^y\varphi)-\frac{1}{2}(\dq^y\varphi)^2=\int_0^1\frac{(1-h)^2}{2}(\dq^y\varphi)^3F'''(h\dq^y\varphi)\,dh.
\end{align*}
When $\alpha<1$, from Moser's estimate (the nonlinear version, as well as the one for products), Lemma \ref{Trilinear integral estimate-v0}, and Sobolev embedding and interpolation, we get that 

\begin{align*}
   \lambda^{\frac{\alpha}{2}-\frac{1}{2}}&\left\|P_\lambda \int \frac{1}{|y|^{\alpha-1}}(\dq^y\varphi)^3F'''(h\dq^y\varphi)\sdq^y\varphi_x \,dy\right\|_{L_x^{\infty}}\\
   &\lesssim \int\frac{1}{|y|^{\alpha-1}}\||D_x|^{\frac{\alpha}{2}}((\dq^y\varphi)^3\sdq^y\varphi_xF'''(h\dq^y\varphi))\|_{L_x^{2}}\,dy\\
   &\lesssim \|\varphi_x\|_{L_x^\infty}\||D_x|^{\frac{\alpha}{2}}\varphi_{x}\|_{L_x^{2}}(\|\varphi_x\|^3_{L_x^\infty}+\|\varphi_x\|^2_{L_x^\infty}\||D_x|^{\delta}\varphi\|_{L_x^\infty})\\
   &\lesssim \frac{1}{t^{2}}\eps^5\langle t\rangle^{C\eps^2}.
   \end{align*}
When $\alpha>1$, we have
\begin{align*}
   \lambda^{\delta}&\left\|P_\lambda \int \frac{1}{|y|^{\alpha-1}}(\dq^y\varphi)^3F'''(h\dq^y\varphi)\sdq^y\varphi_x \,dy\right\|_{L_x^{\infty}}\\
   &\lesssim \int\frac{1}{|y|^{\alpha-1}}\||D_x|^{2\delta}((\dq^y\varphi)^3\sdq^y\varphi_xF'''(h\dq^y\varphi))\|_{L_x^{1/\delta}}\,dy\\
   &\lesssim \|\varphi_x\|^4_{L_x^\infty}(\||D_x|^{\frac{1}{2}+\delta}\varphi_{x}\|_{L_x^{2}}+\||D_x|^{\frac{1}{2}+\delta}\varphi_{xx}\|_{L_x^{2}})\\
   &+\|\varphi_x\|^3_{L_x^\infty}\||D_x|^{\frac{1}{2}+\delta}\varphi_{x}\|_{L_x^{2}}(\|\varphi_x\|_{L_x^\infty}+\||D_x|^{\alpha-1+\delta}\varphi_x\|_{L_x^\infty})\\
   &\lesssim \frac{1}{t^{2}}\eps^5\langle t\rangle^{C\eps^2}.
   \end{align*}

We have also used Sobolev embedding and the classical Moser estimate, keeping in mind $F'''(0)=0$. Similarly,
\begin{align*}
   \left\|\int_{\mathbb{R}}\frac{1}{|y|^{\alpha-1}}P_\lambda((F(\dq^y\varphi)-\frac{1}{2}(\dq^y\varphi)^2)\sdq^y\varphi_x)\,dy\right\|_{L_x^{2}}&\lesssim 
   \lambda^{-\delta}\frac{1}{t^{3/2}}\eps^5\langle t\rangle^{C\eps^2}.
\end{align*}

By H\"older's inequality and Young's inequality respectively,
\begin{align*}
\left|\left\langle P_\lambda \int_{\mathbb{R}}\frac{1}{|y|^{\alpha-1}}\left(F(\dq^y\varphi)-\frac{1}{2}(\dq^y\varphi)^2\right)\sdq^y\varphi_x\,dy,\mathbf{\varphi}_v\right\rangle\right|&\lesssim \lambda^{\max\{\frac{1-\alpha}{2},0\}-\delta}\frac{1}{t^{3/2}}\eps^5\langle t\rangle^{C\eps^2}, \\
\left\|\left\langle P_\lambda\int_{\mathbb{R}}\frac{1}{|y|^{\alpha-1}}\left(F(\dq^y\varphi)-\frac{1}{2}(\dq^y\varphi)^2\right)\sdq^y\varphi_x\,dy,\mathbf{\varphi}_v\right\rangle\right\|_{L_v^2(J_{\lambda})}&\lesssim \lambda^{-\delta}\frac{1}{t^{3/2}}\eps^5\langle t\rangle^{C\eps^2}.\end{align*}

\

We are left to estimate
\begin{align*}
\left\langle P_\lambda\int_{\mathbb{R}}\frac{1}{|y|^{\alpha-1}}(\dq^y\varphi)^2\sdq^y\varphi_x\,dy,\mathbf{u}^v\right\rangle&=\left\langle \partial_x P_\lambda Q^{\text{bal}}(\varphi),\mathbf{u}^v\right\rangle+\left\langle \chi_\lambda^1\partial_x P_\lambda Q^{\text{unbal}}(\varphi),\mathbf{u}^v\right\rangle\\
&\quad +\left\langle (1-\chi_\lambda^1)\partial_x P_\lambda Q^{\text{unbal}}(\varphi),\mathbf{u}^v\right\rangle,
\end{align*}
where $\chi^1_{\lambda}$ be a cut-off function enlarging $\chi_{\lambda}$. Due to the fact that $\displaystyle \mathbf{u}^v$ is supported in the region $\displaystyle \left|\frac{x}{t}-v\right|\lesssim \lambda^{\frac{\alpha}{2}-1}t^{-1/2}$, the condition $(t,v)\in\mathcal{D}$ will imply that the third term is identically zero, while Lemma \ref{Unbalanced cubic estimate} implies that the second term is an acceptable error. 
Thus, we only have to analyze
\begin{align*}
\left\langle \partial_x P_\lambda Q^{\text{bal}}(\varphi),\mathbf{u}^v\right\rangle =\left\langle \partial_x P_\lambda Q(\varphi_\lambda),\mathbf{u}^v\right\rangle.
\end{align*}
 Let $\chi^1$ be a cut-off function that is equal to $1$ on the support of the wave packet $\mathbf{u}_v$. Let $\tilde{\chi}$ be another cut-off function whose support is  slightly larger than the one of $\chi^1$.
 We write
 \begin{align*}
\left\langle \partial_x P_\lambda Q(\varphi_\lambda),\mathbf{u}^v\right\rangle&=\left\langle \partial_x P_\lambda \tilde{\chi}Q(\varphi_\lambda),\mathbf{u}^v\right\rangle+\left\langle \chi^1\partial_x P_\lambda(1-\tilde{\chi})Q(\varphi_\lambda),\mathbf{u}^v\right\rangle
 \end{align*}
 As in the proof of Lemma \ref{Unbalanced cubic estimate}, we note that the operator norm bounds
 \begin{align*}
\|\chi^1\partial_xP_\lambda(1-\tilde{\chi})\|_{L^\infty\rightarrow L^\infty}+\|\chi^1\partial_xP_\lambda(1-\tilde{\chi})\|_{L^2\rightarrow L^2}&\lesssim \lambda^{1-2N}t^{-N}
 \end{align*}
 for every $N$ imply that the second term is acceptable error. This leaves us with the first.
 
We first replace $\varphi_\lambda$ by $\chi_\lambda\varphi_\lambda$. When $\alpha>1$, from Lemma \ref{Cubic difference bound}, we have
\begin{align*}
|\langle \partial_x P_\lambda \tilde{\chi}&(Q(\varphi_\lambda)-Q(\chi_\lambda\varphi_\lambda)),\mathbf{u}^v\rangle|\\
&\lesssim \lambda\|(\partial_x(\chi_\lambda\varphi_\lambda),\partial_x\varphi_\lambda)\|_{L_x^\infty}\|(|D_x|^{\alpha-1}(\chi_\lambda\varphi_\lambda),|D_x|^{\alpha-1}\varphi_\lambda)\|_{L_x^\infty}\|\partial_x((1-\chi_\lambda)\varphi_\lambda)\|_{L_x^{\infty}}\|\mathbf{u}^v\|_{L_x^1}\\
&+\lambda\|(\partial_x(\chi_\lambda\varphi_\lambda),\partial_x\varphi_\lambda)\|
_{L_x^{\infty}}\|(\chi_\lambda\varphi_\lambda,\varphi_\lambda)\|_{L_x^\infty} \|\partial_x((1-\chi_\lambda)\varphi_\lambda)\|_{L_x^{\infty}}\|\mathbf{u}^v\|_{L_x^1},
\end{align*}
while for $\alpha<1$, the same lemma implies that
\begin{align*}
\lambda^{\frac{\alpha-1}{2}}|\langle \partial_x P_\lambda \tilde{\chi}&(Q(\varphi_\lambda)-Q(\chi_\lambda\varphi_\lambda)),\mathbf{u}^v\rangle|\\
&\lesssim \lambda^{\frac{\alpha+1}{2}}\|((\chi_\lambda\varphi_\lambda),\varphi_\lambda)\|^2_{L_x^\infty}\|\partial_x((1-\chi_\lambda)\varphi_\lambda)\|_{L_x^{\infty}}\|\mathbf{u}^v\|_{L_x^1}\\
&+\lambda^{\frac{\alpha+1}{2}}\|(|D_x|^{\alpha+\delta}(\chi_\lambda\varphi_\lambda),|D_x|^{\alpha+\delta}\varphi_\lambda)\|
_{L_x^{\infty}}\|(\chi_\lambda\varphi_\lambda,\varphi_\lambda)\|_{L_x^\infty} \|\partial_x((1-\chi_\lambda)\varphi_\lambda)\|_{L_x^{\infty}}\|\mathbf{u}^v\|_{L_x^1}.
\end{align*}
By interpolation, along with Lemma \ref{Elliptic bounds for the derivative} and the condition $(t,v)\in\mathcal{D}$, it follows that the errors are acceptable.  The $L_x^2$-bound is similar.

\

We now replace $\chi_\lambda\varphi_\lambda$ by $\psi$. From Lemma \ref{Cubic difference bound}, when $\alpha>1$, we have
\begin{align*}
|\langle \partial_x P_\lambda \tilde{\chi}&(Q(\chi_\lambda\varphi_\lambda)-Q(\psi)),\mathbf{u}^v\rangle|\\
&\lesssim \lambda\|(\partial_x(\chi_\lambda\varphi_\lambda),\partial_x\psi)\|_{L_x^\infty}^2\|\partial_x(\chi_\lambda(x/t) r^\lambda)\|_{L_x^{\infty}}\|\mathbf{u}^v\|_{L_x^1}\\
&+\lambda\|(\partial_x(\chi_\lambda\varphi_\lambda),\partial_x\psi)\|
_{L_x^{\infty}}\|(\chi_\lambda\varphi_\lambda,\psi)\|_{L_x^\infty} \|\partial_x(\chi_\lambda(x/t) r^\lambda)\|_{L_x^{\infty}}\|\mathbf{u}^v\|_{L_x^1},
\end{align*}
and
\begin{align*}
\lambda^{\frac{\alpha-1}{2}}|\langle \partial_x P_\lambda \tilde{\chi}&(Q(\chi_\lambda\varphi_\lambda)-Q(\psi)),\mathbf{u}^v\rangle|\\
&\lesssim \lambda^{\frac{\alpha+1}{2}}\|((\chi_\lambda\varphi_\lambda),\psi)\|_{L_x^\infty}^2\|\partial_x(\chi_\lambda(x/t) r^\lambda)\|_{L_x^{\infty}}\|\mathbf{u}^v\|_{L_x^1}\\
&+\lambda^{\frac{\alpha+1}{2}}\|(\partial_x(\chi_\lambda\varphi_\lambda),\partial_x\psi)\|
_{L_x^{\infty}}\|(\chi_\lambda\varphi_\lambda,\psi)\|_{L_x^\infty} \|\partial_x(\chi_\lambda(x/t) r^\lambda)\|_{L_x^{\infty}}\|\mathbf{u}^v\|_{L_x^1},
\end{align*}
when $\alpha<1$.

By interpolation, along with Lemmas \ref{Error bounds 1} and \ref{Psi bounds}, and the condition $(t,v)\in\mathcal{D}$, it follows that the errors are acceptable.  The $L_x^2$-bound is similar.

\

We seek to replace $\psi$ by $\theta$. We evaluate
\begin{align*}
\left\langle \partial_x P_\lambda \tilde{\chi}(Q(\psi)-Q(\theta)),\mathbf{u}^v\right\rangle
\end{align*}
We have
\begin{align*}
\left|\tilde{\chi}(Q(\psi)-Q(\theta))\right|&\lesssim \left|\tilde{\chi}\left(\frac{x-vt}{\sqrt{|ta''(\xi_v)|}}\right)\right|\int\frac{1}{|y|^{\alpha-1}}(|\dq^y\psi|^2+|\dq^y\theta|^2)|\delta^y\beta^\lambda_v(x)|\,dy
\end{align*}
The support condition of $\tilde{\chi}$ implies that $x$ is in the region $|x-vt|\lesssim \delta x=t^{1/2}\lambda^{\frac{\alpha}{2}-1}$. From Lemma \ref{Beta bounds} we now get that
\begin{align*}
\left|\tilde{\chi}(Q(\psi)-Q(\theta))\right|\lesssim t^{-3/4}\lambda^{\frac{\alpha}{4}-\half}\|\varphi\|_X\int\frac{1}{|y|^{\alpha-1}}(|\dq^y\psi|^2+|\dq^y\theta|^2)\,dy 
\end{align*}
Bernstein's inequality and Lemma \ref{Trilinear integral estimate-v0} imply that
\begin{align*}
\left|\partial_xP_\lambda\tilde{\chi}(Q(\psi)-Q(\theta))\right|&\lesssim t^{-3/4}\lambda^{\frac{\alpha}{4}+\half}\|\varphi\|_X(\|\psi_x\|^2_{L_x^\infty}+\|\theta_x\|^2_{L^\infty})\|\mathbf{u}^v\|_{L_x^1}\\
&+t^{-3/4}\lambda^{\frac{\alpha}{4}+\half}\|\varphi\|_X(\|\psi\|_{L_x^\infty}\|\psi_x\|_{L_x^{\infty}}+\|\theta\|_{L_x^\infty}\|\theta_x\|_{L_x^{\infty}})\|\mathbf{u}^v\|_{L_x^{1}},
\end{align*}
when $\alpha>1$, and
\begin{align*}
\lambda^{\frac{\alpha-1}{2}}\left|\partial_xP_\lambda\tilde{\chi}(Q(\psi)-Q(\theta))\right|&\lesssim t^{-3/4}\lambda^{\frac{3\alpha}{4}}\|\varphi\|_X(\|\psi\|^2_{L_x^\infty}+\|\theta\|^2_{L^\infty})\|\mathbf{u}^v\|_{L_x^1}\\
&+t^{-3/4}\lambda^{\frac{3\alpha}{4}}\|\varphi\|_X(\|\psi\|_{L_x^\infty}\|\psi_x\|_{L_x^{\infty}}+\|\theta\|_{L_x^\infty}\|\theta_x\|_{L_x^{\infty}})\|\mathbf{u}^v\|_{L_x^{1}},
\end{align*}
when $\alpha<1$.

From Lemmas \ref{Psi bounds} and \ref{Theta bounds}, along with the condition $(t,v)\in\mathcal{D}$, it follows in both cases that this error is acceptable. The $L_x^2$-bound is similar.

\

We are left to analyze
\begin{align*}
t^{-3/2}\gamma(t,v)|\gamma(t,v)|^2\left\langle \partial_x P_\lambda Q(\chi_\lambda e^{it\phi(x/t)}),\mathbf{u}^v\right\rangle.
\end{align*}
Since by Lemma~\ref{Gamma bounds},
\[
\|t^{-3/2}\gamma(t,v)|\gamma(t,v)|^2\|_{L_v^\infty(J_\lambda)} \lesssim  \|\varphi_\lambda\|^3_{L_x^\infty}, \qquad \|t^{-3/2}\gamma(t,v)|\gamma(t,v)|^2\|_{L_v^2(J_\lambda)} \lesssim t^{-1/2}\|\varphi_\lambda\|^2_{L_x^\infty}\|\varphi_\lambda\|_{L_x^2}
\]
it suffices to estimate
\begin{align*}
|\langle \partial_x P_\lambda Q(\chi_\lambda e^{it\phi(x/t)}),\mathbf{u}^v\rangle - t^{\half}q(\xi_v)\xi_v(\chi_\lambda(v))^3| &\lesssim \lambda^{\max\{\frac{1-\alpha}{2},0\}-\delta}t^{\frac12+C\epsilon^2}\epsilon\\
|\langle \partial_x P_\lambda Q(\chi_\lambda e^{it\phi(x/t)}),\mathbf{u}^v\rangle - t^{\half}q(\xi_v)\xi_v(\chi_\lambda(v))^3| &\lesssim \lambda^{-\delta}t^{\frac12+C\epsilon^2}\epsilon.
\end{align*}

 We note that
    \begin{align*}
\dq^y(\chi_{\lambda}e^{\pm it\phi(x/t)})&=\chi_{\lambda}\dq^y\left(e^{\pm it\phi(x/t)}\right)+\dq^y(\chi_{\lambda})e^{\pm it\phi((x+y)/t)}.
\end{align*}
Lemma \ref{Trilinear integral estimate-v0}, implies that 
\begin{align*}
\left|\left\langle\partial_x P_\lambda\int \frac{1}{|y|^{\alpha-1}}\dq^y(\chi_\lambda)e^{it\phi((x+y)/t)}\dq^y(\chi_\lambda e^{-it\phi(x/t)})\dq^y(\chi_\lambda e^{it\phi(x/t)})\,dy,\mathbf{u}^v \right\rangle\right|&\lesssim \lambda(t^{-1/2}\lambda^{2-\alpha} +t^{-1/2}\lambda^{3-\alpha}), 
\end{align*}
when $\alpha>1$, and 
\begin{align*}
\left|\left\langle\partial_x P_\lambda\int \frac{1}{|y|^{\alpha-1}}\dq^y(\chi_\lambda)e^{it\phi((x+y)/t)}\dq^y(\chi_\lambda e^{-it\phi(x/t)})\dq^y(\chi_\lambda e^{it\phi(x/t)})\,dy,\mathbf{u}^v \right\rangle\right|&\lesssim \lambda(t^{-1/2}\lambda^{1-\alpha} +t^{-1/2}\lambda^{3-\alpha}), 
\end{align*}
when $\alpha<1$.

The most problematic contribution in the case $\alpha>1$ is the one that arises from the first term. We have
\begin{align*}
\lambda^{\delta}\lambda^{3-\alpha}t^{-1/2}\|\varphi_\lambda\|^3_{L_x^\infty}&\lesssim \lambda^{-\alpha}t^{-2}\|\varphi\|^3_X\lesssim t^{-1-\delta}\epsilon^3t^{C\epsilon^2}
\end{align*}
The other term is analogous.

The most problematic contribution in the case $\alpha<1$ is the one that arises from the second term. We have
\begin{align*}
\lambda^{\frac{\alpha+1}{2}+\delta}\lambda^{3-\alpha}t^{-1/2}\|\varphi_\lambda\|^3_{L_x^\infty}&\lesssim t^{-\frac{3}{2}}\|\varphi\|^3_X\lesssim t^{-1-\delta}\epsilon^3t^{C\epsilon^2}
\end{align*}
The other term is analogous.

The $L_v^2$-bound is treated similarly, and so is the case in which one chooses the term $\displaystyle\dq^y(\chi_{\lambda})e^{-it\phi((x+y)/t)}$ in the expansion of $\displaystyle \dq^y(\chi_{\lambda}e^{-it\phi(x/t)})$. This leaves us with
\begin{align*}
\langle \partial_x P_\lambda\left(\chi_\lambda(x/t)^3Q(e^{it\phi(x/t)})\right),\mathbf{u}^v\rangle
\end{align*}
Lemma \ref{Semiclassical computation}, implies that we can replace the latter with
\begin{align*}
\langle \partial_x P_\lambda\left(\chi_\lambda(x/t)^3e^{it\phi(x/t)}(\phi'(x/t))^2q(1)\right),\mathbf{u}^v\rangle,
\end{align*}
 with error bounded by 
\begin{align*}
\lambda t^{1/2} \left(\frac{\lambda^{6-3\alpha}+\lambda^{5-2\alpha}}{t^{2-(4-\alpha)a}}+\frac{\lambda^{3-\alpha}}{t^{1-(2-\alpha)a}}+\frac{1}{t^{(1+\alpha)a}}\right)
\end{align*}
when $\alpha>1$, and
\begin{align*}
\lambda t^{1/2} \left(\frac{\lambda^{4-2\alpha}}{t^{2-(3-\alpha)a}}+\frac{\lambda^{3-\alpha}}{t^{1-(2-\alpha)a}}+\frac{1}{t^{(1+\alpha)a}}\right)
\end{align*}
when $\alpha<1$.

We note that one problematic contribution is the one arising from the last term. From $\displaystyle(t,v)\in\mathcal{D}$, we have the bound
\begin{align*}
\frac{\lambda^{1+\delta}}{t^{(1+\alpha)a-1/2}}\|\varphi_\lambda\|^3_{L_x^\infty}&\lesssim \lambda^{-2}t^{-(1+\alpha)a-1}\|\varphi\|^3_X\lesssim \lambda^{-1-\delta}t^{-1-\delta}\epsilon^3 t^{C\epsilon^2}
\end{align*}
We note that this contribution is acceptable. We now take $a=\frac{1}{10}$. The only other problematic contribution is the one arising from the first term, for which we bound
\begin{align*}
\lambda^{\delta}\frac{\lambda^{7-3\alpha}+\lambda^{6-2\alpha}}{t^{3/2-(4-\alpha)a}}\|\varphi_\lambda\|^3_{L_x^\infty}&\lesssim (\lambda^{4-3\alpha-\delta}+\lambda^{3-2\alpha-\delta})t^{(4-\alpha)a-3}\|\varphi\|^3_X\lesssim (\lambda^{4-3\alpha-\delta}+\lambda^{3-2\alpha-\delta})t^{-1-\delta}\epsilon^3 t^{C\epsilon^2},
\end{align*}
when $\alpha>1$, and
\begin{align*}
\lambda^{2\alpha-3+\delta}\frac{\lambda^{5-2\alpha}}{t^{3/2-(3-\alpha)a}}\|\varphi_\lambda\|^3_{L_x^\infty}&\lesssim t^{(3-\alpha)a-3}\|\varphi\|^3_X\lesssim t^{-6/5}\epsilon^3 t^{C\epsilon^2},
\end{align*}
when $\alpha<1$.

The contribution arising from the second term can be immediately bounded by
\begin{align*}
\frac{\lambda^{\min\{\frac{7-\alpha}{2},4-\alpha\}+\delta}}{t^{1/2-(2-\alpha)a}}\|\varphi_\lambda\|^3_{L_x^\infty}&\lesssim t^{(2-\alpha)a-\frac{3}{2}}\|\varphi\|^3_X\lesssim t^{-1-\delta}\epsilon^3 t^{C\epsilon^2}
\end{align*}
The $L_v^2$-bound is similar.
This means that we have to analyze
\begin{align*}
&q(1)\langle \partial_x (\chi_\lambda(x/t)^3(\phi'(x/t))^2e^{it\phi(x/t)}),\mathbf{u}^v_\lambda \rangle =q(1)\langle (\chi_\lambda(x/t)\phi'(x/t))^3e^{it\phi(x/t)},\mathbf{u}^v_\lambda \rangle\\
&\quad +q(1)t^{-1}\langle (3\chi_\lambda(x/t)^2\chi_\lambda'(x/t)(\phi'(x/t))^2+2\chi_\lambda(x/t)^3\phi'(x/t)\phi''(x/t)e^{it\phi(x/t)}),\mathbf{u}^v_\lambda \rangle,
\end{align*}
where the last contribution can be immediately shown to be an acceptable error by using the condition $(t,v)\in\mathcal{D}$. Further, we may replace $\pax^v_\lambda$ by $\pax^v$. To see this, from the proof of Lemma 5.8 in \cite{ITpax}, we have
\begin{align*}
    \left|P_{\neq\lambda}\mathbf{u}^v\right|&\lesssim \lambda^{1-\frac{\alpha}{2}}(1+|y|)^{-1-\delta}t^{-1-\delta}\lambda^{-(1+\delta)\alpha}, \qquad y=(x - vt)|ta''(\xi_v)|^{-\half},
\end{align*}
and
\begin{align*}
|(\chi_\lambda(x/t)\phi'(x/t))^3e^{it\phi(x/t)}|&\lesssim \lambda^3.
\end{align*}
Thus,
\begin{align*}
\left|\left\langle (\chi_\lambda(x/t)\phi'(x/t))^3e^{it\phi(x/t)},P_{\neq\lambda}\mathbf{u}^v\right\rangle\right|&\lesssim \lambda^3\lambda^{1-\frac{\alpha}{2}}\lambda^{-(1+\delta)\alpha}t^{-1-\delta}t^{1/2}\lambda^{\alpha/2-1}\lesssim t^{-1/2-\delta}\lambda^{3-(1+\delta)\alpha},
\end{align*}
which along with the condition $(t,v)\in\mathcal{D}$ shows that this is an acceptable error.

As $\mathbf{u}^v$ is supported in the region $\displaystyle\left|\frac{x}{t}-v\right|\lesssim t^{-1/2}\lambda^{\frac{\alpha}{2}-1}$, we can replace $x/t$ by $v$ in $\chi_\lambda(x/t)\phi'(x/t)$, with acceptable errors. As $\chi_\lambda(v)=1$, the remaining term is now
\begin{align*}
iq(1)(\chi_\lambda(v)\xi_v)^3\left\langle e^{it\phi(x/t)},\mathbf{u}^v\right\rangle &=t^{\half}iq(\xi_v)\xi_v,
\end{align*}
as desired.
\end{proof}
\subsection{Closing the bootstrap argument}
We recall that
\begin{align*}
  \|\varphi_\lambda\|_{L_x^{\infty}}
  &\lesssim\frac{1}{\sqrt{t}}\lambda^{-1+\delta+\delta_1}\|\varphi\|_X\lesssim \frac{1}{\sqrt{t}}\lambda^{-1+\delta+\delta_1}\epsilon t^{C\epsilon^2}
\end{align*}
and when $\lambda>1$,
\begin{align*}
  \|\varphi_\lambda\|_{L_x^{\infty}}
  &\lesssim \frac{1}{\sqrt{t}}\lambda^{-(1+\alpha/2+2\delta)}\|\varphi\|_X\lesssim \frac{1}{\sqrt{t}}\lambda^{-(1+\alpha/2+2\delta)}\epsilon t^{C\epsilon^2}.
\end{align*}

Thus, if $\displaystyle t\lesssim \lambda^{N}$ when $\lambda>1$, and if $\displaystyle t\lesssim \lambda^{-N}$ when $\lambda\leq 1$, where $N$ can be chosen arbitrarily, we get the desired bounds. We are left to analyze $\displaystyle t\gtrsim \lambda^{N}$ when $\lambda>1$, and $\displaystyle t\gtrsim \lambda^{-N}$ when $\lambda\leq 1$.

We recall the following bounds in the elliptic region:
\begin{align*}
   \||D_x|^{1-\delta-\delta_1}((1-\chi_{\lambda})\varphi_\lambda(x))\|_{L_x^{\infty}}&\lesssim \frac{\lambda^{\frac32-\alpha-\delta-\delta_1}}{t}(\|\varphi\|_X+\|\varphi_\lambda\|_{L_x^2})\lesssim \frac{\lambda^{\frac32-\alpha-\delta-\delta_1}+\lambda^{1-\alpha-\delta-\delta_1}}{t}\epsilon t^{C\epsilon^2}\\
\||D_x|^{\frac{\alpha}{2}+\delta}\partial_x((1-\chi_{\lambda})\varphi_\lambda(x))\|_{L_x^{\infty}}&\lesssim \frac{\lambda^{\frac{3-\alpha}{2}+\delta}}{t}(\|\varphi\|_X+\|\varphi_{\lambda}\|_{L_x^2})\lesssim \frac{\lambda^{\frac{3-\alpha}{2}+\delta}+\lambda^{\frac{2-\alpha}{2}+\delta}}{t}\epsilon t^{C\epsilon^2}
\end{align*}
which gives the desired bounds when $t\gtrsim \lambda^{N}$ ($\lambda>1$), and $t\gtrsim \lambda^{-N}$ ($\lambda\leq 1$).
We still have to bound $\displaystyle \chi_{\lambda}\varphi_\lambda$.
We recall that, if $x/t\in J_{\lambda}$, and $\displaystyle r(t,x)=\chi_{\lambda}\varphi_\lambda(t,x)-\frac{1}{\sqrt{t}}\chi_{\lambda}\gamma(t,x/t)e^{it\phi(x/t)}$,
\begin{align*}
   t^{1/2}\|r^\lambda\|_{L_x^{\infty}}&\lesssim t^{-1/4}\lambda^{\frac{\alpha}{4}-\frac32}\eps t^{C\eps^2} 
\end{align*}

We note that
\begin{align*}
  t^{-1/4}\lambda^{\frac{\alpha}{4}-\frac32}\eps t^{C\eps^2}&\lesssim \lambda^{1-\delta-\delta_1}\eps   \end{align*}
when $\lambda\leq 1$, because this is equivalent to
\begin{align*}
    \lambda^{\frac{\alpha}{4}-\frac{5}{2}+\delta+\delta_1}\lesssim t^{1/4-C\eps^2}
\end{align*}
(this is true when $t\gtrsim \lambda^{-N}$),
and that
\begin{align*}
  t^{-1/4}\lambda^{\frac{\alpha}{4}-\frac32}\eps t^{C\eps^2}&\lesssim \lambda^{-\left(1+\frac{\alpha}{2}+3\delta/2\right)}\eps   
\end{align*}
when $\lambda>1$, because this is equivalent to
\begin{align*}
    \lambda^{\frac{3\alpha}{4}-\frac12+3\delta/2}\lesssim t^{1/4-C\eps^2}
\end{align*}
(this is true when $t\gtrsim \lambda^{N}$).

This means that we only need the bounds
\begin{align*}
   |\gamma(t,v)|\lesssim\eps \lambda^{-\left(1-\delta-\delta_1\right)} 
\end{align*}
when $\lambda\leq 1$, and
\begin{align*}
    |\gamma(t,v)|\lesssim\eps \lambda^{-\left(1+\frac{\alpha}{2}+3\delta/2\right)} 
\end{align*}
when $\lambda>1$.
By initializing at time $t=1$, up to which the bounds are known to be true from the energy estimates, and by using Proposition \ref{Asymptotic equation}, we reach the desired conclusion.

\section{Modified scattering}\label{s:scattering}

In this section, we prove the final part of Theorem \ref{t:gwp}, which refers to the modified scattering behavior of the solutions constructed in Section~\ref{s:gwp}.

 As was shown by Hunter-Shu-Zhang \cite{HSZfamily} when $\alpha<1$, the mass of the solutions to \eqref{gSQG} stays conserved. We begin with a short proof of this property for the full range $(0,2)-\{1\}$ :
\begin{proposition}\label{Conservation of mass}
   For solutions $\varphi$ of \eqref{gSQG}, $\|\varphi(t)\|_{L^2}$ is conserved in time.  
\end{proposition}
\begin{proof}
We have
\begin{align*}
  \frac{d}{dt}\|\varphi\|^2_{L_x^2}&=\int\varphi_t\cdot\varphi\,dx\\
  &=-2\int\int\frac{1}{|y|^{\alpha-1}} F(\dq^y\varphi)\sdq^y\varphi_x\,dy\cdot\varphi\,dx-2c(\alpha)\int \varphi\cdot |\partial_x|^{\alpha-1}\varphi_x\,dx\\
  &=-2\int\int\frac{1}{|y|^{\alpha-1}} F(\dq^y\varphi)\sdq^y\varphi_x\,dy\cdot\varphi\,dx:=-2I.
\end{align*}

We note that by the change of variables $(x, y) \mapsto (x + y, -y)$,
\begin{align*}
I&= -\int\int \frac{1}{|y|^{\alpha-1}}F(\dq^{y}\varphi)\sdq^{y}\varphi_x\cdot\varphi(x+y)\,dx\,dy.
\end{align*}
Thus,
\begin{align*}
-2I&=\int\int \frac{1}{|y|^{\alpha-1}}F(\dq^{y}\varphi)\sdq^{y}\varphi_x\cdot(\varphi(x+y) - \varphi(x))\,dx\,dy\\
&=\int|y|^{2-\alpha}\int F(\dq^{y}\varphi)\dq^y\varphi\cdot \dq^{y}\varphi_x\,dx\,dy\\
&=\int|y|^{2-\alpha}\int \partial_x(G(\dq^{y}\varphi))\,dx\,dy=0,
\end{align*}
where $\displaystyle G(x)=\frac{x^2}{2}-\frac{1}{2-\alpha}(1+x^2)^{1-\frac{\alpha}{2}}$.
\end{proof}
Recall the asymptotic equation
\begin{align*}
    \dot{\gamma}(t,v)&=iq(\xi_v)\xi_v t^{-1}\left|\gamma(t,v)\right|^2\gamma(t,v)+f(t,v),
\end{align*}

As $t\rightarrow\infty$, $\gamma(t,v)$ converges to the solution of the equation
\begin{align*}
    \dot{\tilde{\gamma}}(t,v)&=iq(\xi_v)\xi_vt^{-1}\tilde{\gamma}(t,v)|\tilde{\gamma}(t,v)|^2,
\end{align*}
whose solution is 
\begin{align*}
    \tilde{\gamma}(t,v)&=W(v)e^{iq(\xi_v)\xi_v\ln(t)|W(v)|^2}
\end{align*}
We can immediately see that
$W(v)$ is well-defined, as
$|W(v)|=|\tilde{\gamma}(t,v)|$, which is a constant, and
\[
W(v)=\lim_{\substack{s\rightarrow\infty}}\tilde{\gamma}(e^{2s\pi/(q(\xi_v)\xi_v|W(v)|^2)},v).
\]

\begin{corollary}\label{Asymptotic expansions}
Let $v\in J_\lambda$. Under the assumption $(t,v)\in\mathcal{D}$, we have
\begin{align*}
    \|\varphi_0\|_{X}\lesssim\eps\ll 1,
\end{align*}
as well as $t\gtrsim \lambda^{-N}$ when $\lambda\leq 1$, we have the asymptotic expansions
\begin{equation}\label{Wdiff1}
    \|\gamma(t,v) - W(v)e^{ i q(\xi_v)\xi_v\log t |W(v)|^2}\|_{L^\infty(J_\lambda)} \lesssim \lambda^{-\delta}g(\lambda,\alpha)t^{-\delta+C^2\eps^2}\eps.
\end{equation}
\begin{equation}\label{Wdiff2}
    \|\gamma(t,v) - W(v)e^{ i q(\xi_v)\xi_v\log t |W(v)|^2}\|_{L^2(J_\lambda)} \lesssim \lambda^{-\delta}(1+\lambda^{-3/2})t^{-\delta+C^2\eps^2}\eps.
\end{equation}
\end{corollary}
\begin{proof}
This is an immediate consequence of Proposition \ref{Asymptotic equation}.
\end{proof}
\begin{proposition}
Under the assumption
\begin{align*}
    \|\varphi_0\|_{X}\lesssim\eps\ll 1,
\end{align*}
the asymptotic profile $W$ defined above satisfies \begin{align*}
\|(-v)^{\frac{1+\delta}{\alpha-1}}(-v)^{\frac{\sgn(\log|\xi_v|)}{\alpha-1}(1+\delta/2)}|D_v|^{1-C_1\epsilon^2}W(v)\|_{L_v^2}\lesssim\epsilon
\end{align*}
Moreover, when $s_0=0$, we also have $\|W(v)\|_{L_v^2}\lesssim\eps$.
\end{proposition}
\begin{proof}
We fix $\lambda$, and let $t\gtrsim\max\{1,\lambda^{-\alpha}\}:=t_\lambda$. From Corollary \ref{Asymptotic expansions} we know that
\begin{align*}
\|W(v)-e^{- i q(\xi_v)\xi_v\log t |\gamma(t,v)|^2}\gamma(t,v)\|_{L_v^2(J_{\lambda})}\lesssim \lambda^{-\delta}(1+\lambda^{-3/2})t^{-\delta+C^2\eps^2}\eps
\end{align*}
From the product and chain rules with Lemma \ref{Gamma bounds}, we have 
\begin{align*}
    \left\|\partial_v\left(e^{-i q(\xi_v)\xi_v\log t |\gamma(t,v)|^2}\gamma(t,v)\right)\right\|_{L_v^2(J_{\lambda})}&\lesssim \lambda^{-\delta}(1+\lambda^{-2})\log(t)\eps t^{C^2\eps^2}.
\end{align*}
In this case,
\begin{align*}
W(v)=O_{\dot{H}^1_v(J_\lambda)}(\lambda^{-\delta}(1+\lambda^{-2})\log(t)\eps t^{C^2\eps^2})+O_{L_v^2(J_\lambda)}(\lambda^{-\delta}(1+\lambda^{-3/2})t^{-\delta+C^2\eps^2}\eps), \qquad t \gtrsim t_\lambda.
\end{align*}
By interpolation  this will imply that
for $C_1$ large enough we have
\begin{align*}
\|W(v)\|_{\dot{H}_v^{1-C_1\epsilon^2}(J_\lambda)}&\lesssim \lambda^{-\delta}(1+\lambda^{-2})\epsilon.
\end{align*}

By dyadic summation over $\lambda\geq 1$ and $\lambda\leq 1$,
\begin{align*}
\|(-v)^{\frac{1+\delta}{\alpha-1}}(-v)^{\frac{\sgn(\log|\xi_v|)}{\alpha-1}(1+\delta/2)}|D_v|^{1-C_1\epsilon^2}W(v)\|_{L_v^2}\lesssim\epsilon
\end{align*}
The last part immediately follows from the conservation of mass.
\end{proof}

\appendix

\section{Euler fronts}\label{s:appendix}
Here we will briefly discuss the ingredients that are necessary in proving the local well-posedness result in the case of Euler fronts, $\alpha=0$. For the purpose of this section, we define our control parameter $B$ as $B:=\|\varphi_x\|_{C^{0,\delta}}$.
\begin{lemma}
We have
\[
Q(\varphi, v) = R(x, D) v
\]
where 
\[
\|(\D_x R)(x, D) v\|_{L^2} \lesssim_A B^2.
\]
\end{lemma}

\begin{proof}
We write
\begin{equation}
\begin{aligned}
Q(\varphi, v) &= \int F(\dq^y \varphi)  \cdot (v(x + y) - v(x)) \, dy
\end{aligned}
\end{equation}
and set
\[
r(x, \xi) = -\int F(\dq^y \varphi) (e^{i\xi y} - 1) \, dy.
\]

We have
\[
\D_x r(x, \xi) =  -\int F'(\dq^y \varphi)\dq^y \varphi_x  (e^{i\xi y} - 1) \, dy,
\]
hence by Lemma~\ref{Trilinear integral estimate-v0} ,\[
|\D_x r| \lesssim_A \int |\dq^y \varphi||\dq^y \varphi_x| \, dy \lesssim_A B^2.
\]

\end{proof}

We now establish energy estimates for the paradifferential equation \eqref{paradiff-eqn}. We define the usual energy
\[
E(v) := \int v^2 \, dx.
\]

\begin{proposition}\label{Paradiferential flow linearized energy estimates-2}
We have
\begin{equation}
\frac{d}{dt} E(v) \lesssim_A B^2 \|v\|_{L^2}^2 + \|f\|_{L^2} \|v\|_{L^2}.
\end{equation}

\end{proposition}

\begin{proof}
Without loss of generality we assume $f = 0$.  Using the equation \eqref{paradiff-eqn} for $v$, we have

\begin{align*}
    \frac{d}{dt}E^0_{\text{lin}}(v)&=2\int v_t\cdot v\,dx=\int (\mathbf{H}v+2\partial_xT_{R} )v\cdot v\,dx\\
    &=\int T_{\partial_xR} v\cdot v\,dx
\end{align*}

This can be immediately seen to satisfy the desired estimate. 

\end{proof}

We now proceed to establish higher order energy estimates. 

\begin{proposition}\label{High energy normalized variable-2}
    Let $s\geq 0$. Given $v$ solving \eqref{paradiff-eqn-inhomog}, there exists a normalized variable $v^s$ such that
    \begin{align*}
        \partial_tv^s - \partial_xQ(\varphi,v^s)-\frac{1}{2}\mathbf{H}v^s= f + \mathcal{R}(\varphi, v),
    \end{align*}
    with
    \begin{align*}
        \|v^s\|_{L_x^2}&\approx \||D_x|^sv\|_{L_x^2}
    \end{align*}
    and $\mathcal{R}(\varphi,v)$ satisfying balanced cubic estimates,
\begin{equation}
\|\mathcal{R}(\varphi, v)\|_{L^2} \lesssim_A B^2 \|v\|_{L^2}.
\end{equation}
\end{proposition}

\begin{proof}
Let $v$ satisfy \eqref{paradiff-eqn-inhomog}, where without loss of generality, $f = 0$. Then, we can immediately see that $v^s := |D_x|^s v$ satisfies
\begin{equation}
\begin{aligned}
\D_t v^s &- \frac{1}{2}\mathbf{H}v-\D_x T_{R}v^s=\mathcal R.
\end{aligned}
\end{equation}

\end{proof}

\

We thus obtain the following energy estimate: 

\begin{proposition}
Let $s\geq 0$. There exist energy functionals $E^s(\varphi)$ such that we have the following:
\begin{enumerate}
    \item[a)]Norm equivalence:
    \begin{align*}
        E^s(\varphi)\approx\|\varphi\|^2_{\dot{H}_x^s}
    \end{align*}
    \item[b)]Energy estimates:
    \begin{align*}
        \frac{d}{dt}E^s(\varphi)\lesssim B^2\|\varphi\|^2_{\dot{H}_x^s}
    \end{align*}
\end{enumerate}
\end{proposition}
\begin{proof}
Let $E^s(\varphi)=E(v^s)$, where $E(v)$ is defined in Proposition \ref{Paradiferential flow linearized energy estimates-2}, and $v^s$ is defined in Proposition \ref{High energy normalized variable-2}. Part a) is immediate, whereas part b) follows from Proposition \ref{Paradiferential flow linearized energy estimates-2}.
\end{proof}

We also have the linearized counterpart:

\begin{proposition}
There exists an energy functional $E^{\text{lin}}(v)$ such that we have the following:
\begin{enumerate}
    \item[a)]Norm equivalence:
    \begin{align*}
        E^{\text{lin}}(v)\approx\|v\|^2_{L_x^2}
    \end{align*}
    \item[b)]Energy estimates:
    \begin{align*}
        \frac{d}{dt}E^{\text{lin}}(v)\lesssim B^2\|v\|^2_{L_x^2}
    \end{align*}
\end{enumerate}
\end{proposition}
\begin{proof}
Let $E^{\text{lin}}(v)=E(v)$, where $E^0(v)$ is defined in Proposition \ref{Paradiferential flow linearized energy estimates-2}. Part a) is immediate, whereas part b) follows from Proposition \ref{Paradiferential flow linearized energy estimates-2}.
\end{proof}

Now the local well-posedness result easily follows as in the other cases.

\bibliography{bib-gsqg}
\bibliographystyle{plain}

\end{document}